\long\global\def\C#1\F{{}}

\C
ENCORE A FAIRE :
        
 ??
   
DECISIONS PRISES :

    -  A \cup B  ou bien A, B ?  Le second !!

    - \geqslant plutot que \ge,  \leqslant plutot que \le !!
    
    - elements OU alternatives ?  LE PREMIER !!!
    
    - facet defining avec ou sans - ?   AVEC !!!
       
    - facet of / for ?   LE PREMIER !!!
    
    - \mid    ou plutot   \st  ?  LE SECOND !!!
    
    - on n'utilise pas \ell !!!
    
    - Inequality~(1)  OU  Inequality in Equation~(1)  
    - Equation~(1) defines a FDI  OU  Inequality in Equation~(1) is a FDI                       LES PREMIERS !!!
    
    - pair of elements PLUTOT QUE arc !!!
    
    - on abrge toujours vertex disjoint en disjoint !!! 
    
    - "nonvalid" s'Žcrit ainsi
    
    - linear ordering polytope plut™t que linear order polytope
 mais par contre linear order tout comme interval order (mme si c'est un peu flottant)

    - on Žcrit "primary linear inequality" avec tous ces mots !!
    
    - on Žcrit characterize plut™t que characterise !!
               characterization plut™t que characterisation !!
\F

\documentclass[authoryear,preprint,12pt]{article} 
\usepackage[longnamesfirst,round]{natbib}
\usepackage{amsmath, amsthm, amssymb, enumerate, color}

\usepackage{tikz}
\usetikzlibrary{decorations,decorations.pathmorphing,decorations.pathreplacing}
\usetikzlibrary{arrows,backgrounds,calc}
\usetikzlibrary{shapes.symbols,shapes.geometric}

\usetikzlibrary{decorations.markings}
 \tikzset{->-/.style={decoration={
  markings,
  mark=at position #1 with {\arrow{triangle 45}}}, 
  postaction={decorate}}}

\newcommand{\R}{\mathbb{R}}

\DeclareMathOperator{\conv}{\operatorname{conv}}

\def\st{{\;\vrule height8pt width1pt depth2pt\;}}

\def\PPP{\mathcal{P}}
\def\RRR{\mathcal{R}}

\def\es{\varnothing}

\newcommand{\plo}[1]{\mathrm{P}_{\mathrm{LO}}^{#1}} 
\newcommand{\pso}[1]{\mathrm{P}_{\mathrm{SO}}^{#1}} 
\newcommand{\pio}[1]{\mathrm{P}_{\mathrm{IO}}^{#1}} 
\newcommand{\ppo}[1]{\mathrm{P}_{\mathrm{PO}}^{#1}} 
\newcommand{\pswo}[1]{\mathrm{P}_{\mathrm{SWO}}^{#1}} 
\newcommand{\pwo}[1]{\mathrm{P}_{\mathrm{WO}}^{#1}} 

\newcommand{\pany}[1]{\mathrm{P}_{\PPP}^{#1}} 

\newcommand{\Vset}[1]{[#1]}
\def\Arcs#1{A_{#1}{}}
\def\RR{\mathbin{R}{}} 
\def\SS{\mathbin{S}{}} 

\theoremstyle{plain}
\newtheorem{prop}{Proposition}
\newtheorem{coro}{Corollary}
\newtheorem{theo}{Theorem}
\newtheorem{lemm}{Lemma}

\newtheorem{prob}{Problem}

\theoremstyle{definition}

\newtheorem{exam}{Example}

\date{May 19, 2014}

\begin{document}


\begin{center}
{
\textbf{Primary Facets of Order Polytopes}%
\footnote{Universit\'e Libre de Bruxelles, D\'epartement de Math\'ematique c.p.~216, B-1050 Bruxelles, Belgium. \{doignon,Selim.Rexhep\}@ulb.ac.be}
}

\vskip3mm

\textsc{Jean-Paul Doignon and Selim Rexhep}
\end{center}

\vskip3mm

\begin{quotation} \small
Mixture models on order relations play a central role in recent investigations of transitivity in binary choice data.
In such a model, the vectors of choice probabilities are the convex combinations of the characteristic vectors of all order relations of a chosen type. 
The five prominent types of order relations are linear orders, weak orders, semiorders, interval orders and partial orders.  For each of them, the problem of finding a complete, workable characterization of the vectors of probabilities is crucial---but it is reputably inaccessible.  
Under a geometric reformulation, the problem asks for a linear description of a convex polytope whose vertices are known.   
As for any convex polytope, a shortest linear description comprises one linear inequality per facet.  Getting all of the facet-defining inequalities of any of the five order polytopes seems presently out of reach.   Here we search for the facet-defining inequalities which we call primary because their coefficients take only the values $-1$, $0$ or $1$.  We provide a classification of all primary, facet-defining inequalities of three of the five order polytopes.   
Moreover, we elaborate on the intricacy of the primary facet-defining inequalities of the linear order and the weak order polytopes.
\end{quotation}

\noindent\textbf{Keywords:}
semiorder, \quad  semiorder polytope, \quad order polytope, \quad facet-defining inequality.

\noindent\textbf{MSC Classification:} [2010] 06A07, \quad 52B12, \quad 91E99.




\section{Introduction}
\label{sect_Introduction}

As it is the case in general for testing a deterministic theory on random sample data \citep{Luce_1959, Luce_1995, Luce_1997}, 
checking whether transitivity is confirmed by
a collection of two-item comparisons raises several interesting issues.
Binary choice data usually consist of the relative frequencies of choice among any two alternatives.
A formal approach to test whether the relative frequencies are consistent with transitivity relies on probabilistic models and derived statistical tests.  \cite{Regenwetter_Davis_Stober_2008} and \cite{Marley_Regenwetter_2016} survey models for binary choice (forced or non-forced).  We focus here on the characterization problem of the choice probabilities predicted by five of the main models.

A random utility model of binary choice relates the probability of chosing alternative $i$ over $j$ to the probability that the utility of $i$, taken as a random variable, exceeds that of $j$.  
As known since a long time for the direct comparison of random utility values\footnote{under the assumption that equality of the utilities of two distinct alternatives occur with probability zero},
the model happens to be a mixture model on linear orders \citep{Block_Marschak_1960}.  In precise terms, vectors of binary choice probabilities coincide with convex combinations of the characteristic vectors of linear orders on the set of alternatives. 
Recent work (\citealp{Regenwetter_Marley_2001, Regenwetter_Davis_Stober_2011, Regenwetter_Davis_Stober_2012}), extends the traditional setting of linear orders to various types of order relations, principally weak orders, semiorders, interval orders and partial orders (the meaning of the terms will be explained in the next section).  The random utility model, based on a specific, modified way of comparing two random utility values, admits a reformulation as a mixture model of order relations.

One of the fundamental problems on probabilistic models is to find out a workable characterization of the (probabilistic) predictions it makes.  In the case of the mixture models of order relations, the characterization plays an important role in implementing tests of  transitivity on binary choice data.
However, complete characterizations were obtained only when the number of alternatives is small (we give details in Section~\ref{sect_diff}).
Even for the particular case of linear orders (in a way the most structured of our relations), the characterization problem raises  mathematical difficulties and is even seen as unsolvable. 
Section~\ref{sect_diff} recalls some explanations for the latter assertion, and moreover indicates why similar difficulties appear for the other four types of order relations.

For mixture models, as the ones we investigate here, a geometric point of view is most useful.  Indeed, a characterization of the model is akin to the description of a certain polytope.  More precisely, the polytope is given by its vertices (in our cases, the characteristic vectors of the order relations) and the aim is to describe the polytope as the solution set of a system of linear (in)equalities.  If such a linear description is moreover a shortest one, then the number of linear equations is equal to the codimension of the polytope, and there is one linear inequality per facet of the polytope.  This shows the importance of facet-defining inequalities, or FDI's.  We refer the reader to Section~\ref{sect_Background} for a short summary of the concepts and results we will need and to \cite{Ziegler_1998} for more background on general convex polytopes.
The five types of order relations we mentioned before lead to five convex polytopes: the linear ordering polytope, the weak order polytope, the interval order polytope, the semiorder polytope and the partial order polytope.  In the past, the first polytope received the most attention (we provide references in the sections dedicated to the respective polytopes but we want to mention here an unpublished manuscript of \citealp{Suck_1995}, the first to promote a common approach to order polytopes).
It is recognized that obtaining a full, linear description of any of these five polytopes is a very difficult problem (cf.~Section~\ref{sect_diff}).  We found it interesting to investigate the facet-defining inequalities with coefficients in the set $\{-1,0,1\}$, with the aim of assessing the relative difficulties in the five cases.  A linear inequality is \textsl{primary} when its coefficients (including the independent term) take only the values $-1$, $0$ or $1$. 
Here is a summary of our results.
  
As a warm-up exercice, we provide a complete description of the primary linear inequalities which define facets of the partial order polytope (Section~\ref{sect_Primary_FDI_PPO}) or the interval order polytope (Section~\ref{sect_Primary_FDI_PIO}).  Then we present a rather satisfiable understanding of the FDI's of the semiorder polytope; here, the results turn out to be rather technical (see Sections~\ref{sect_Lifting_Lemma} to \ref{sec_FDI_PSO}).  
We see the cases of the strict weak order and the linear ordering polytopes to be out of our reach even for primary linear inequalities, as we explain in Sections~\ref{sect_Primary_FDI_PWO} and \ref{sect_Primary_FDI_PLO}. 

Two directions of possible further research are worth mentioning here.  First, techniques from combinatorial optimization could be applied to the primary linear inequalities found here to derive more inequalities (such as those resulting from so-called Chv\'atal-Gomory cuts; for an introduction to the techniques, see \citealp{Bertsimas_Weismantel_2005}, Section~9.4, or \citealp{Conforti_Cornuejols_Zambelli_2014}, Chapter~5). 
Second, when a polytope $Q$ contains a polytope $P$, new FDI's of one of the two polytopes can sometime be infered from FDI's of the other polytope; we leave for future work the related inspection of the inclusions among our five order polytopes.

The authors thank Samuel Fiorini for helpful discussions at the start of the project.

\section{Background: Types of Order Relations and their Polytopes}
\label{sect_Background}

In this section we briefly recall some basic facts, first about order  relations, then about polytopes.
Throughout the paper, $n$ denotes a natural number with $n \geqslant 2$.  We write $\Vset{n}$ for the set $\{1,\ldots,n\}$ of \textsl{elements} (or alternatives).
Moreover, we denote by $\Arcs{n}$ the set of pairs
of distinct elements, that is: 
$$
\Arcs{n} = \{(i,j) \st  i,j \in \Vset{n}, i \neq j\}.
$$

Let $\RR$ be an irreflexive binary relation on $\Vset{n}$ (that is, $\RR \subseteq \Arcs{n}$); we write $i \RR j$ for $(i,j)\in \RR$.  Then $\RR$ is a \textsl{(strict) partial order} if $\RR$ is asymmetric (that is, if $i \RR j$ then not~$j \RR i$) and transitive (if $i \RR j$ and $j \RR k$ then $i \RR k$). 
A \textsl{linear order} is a partial order which is total (two distinct elements are always comparable). 
A \textsl{strict weak order} is a partial order which is \textsl{negatively transitive} (that is, $i \RR k$ implies $i \RR j$ or $j \RR k$).  (Notice that strict weak orders are the complements of `complete preorders', as we explain in Section~\ref{sect_Primary_FDI_PWO}.)
An \textsl{interval order} $S$ on $\Vset{n}$ is a partial order for which there exist two maps $f$ and $g$ from $\Vset{n}$ to $\R$ such that
$$
i \SS j \text{ if and only if } g(i) < f(j) 
$$
(thus $i \SS j$ exactly if the closed interval $[f(i),g(i)]$ of the real line is located entirely below the similar interval $[f(j),g(j)]$).
The pair $(f,g)$ of maps is then called an \textsl{interval representation} of $S$.  If $S$ admits an interval representation $(f,g)$ such that $g(i) = f(i) + 1$ for each $i$ in $\Vset{n}$ (every interval has length $1$) then $S$ is also called a \textsl{semiorder}; we then say that $f$ is a \textsl{unit interval representation}.  We now state two classical theorems characterizing interval orders and semiorders (for the proofs as well as additional basic terminology, see a textbook as for example \cite{Fishburn_1985} or  \cite{Trotter_1992}).
 It is easy to check that the partial orders represented by their Hasse diagrams in Figure~\ref{Forbidden_posets} are not semiorders; we denote them by $\underline{2} + \underline{2}$ and $\underline{3} + \underline{1}$ respectively. The second one is an interval order, while the first one is not.

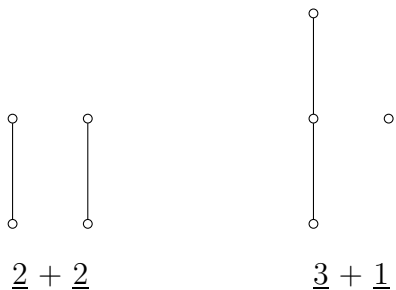
\begin{figure}
\begin{center}
\begin{tikzpicture}[yscale=0.7]
  \tikzstyle{vertex}=[circle,draw,fill=white, scale=0.3]
  
  \node (z_1) at (0,0) [vertex] {};
  \node (z_2) at (0,2) [vertex] {};
  \node (z_3) at (1,0) [vertex] {};
  \node (z_4) at (1,2) [vertex] {};
    \draw (z_1) -- (z_2);
    \draw (z_3) -- (z_4);
\node at (0.5,-1) {\underline{2} + \underline{2}};

\begin{scope}[xshift=4cm]
  \node (z_5) at (0,0) [vertex] {};
  \node (z_6) at (0,2) [vertex] {};
  \node (z_7) at (0,4) [vertex] {};
  \node (z_8) at (1,2) [vertex] {};
     \draw (z_5) -- (z_6) -- (z_7); 
\node at (0.5,-1) {\underline{3} + \underline{1}};
\end{scope}

\end{tikzpicture}
\end{center}
\caption{The Hasse diagrams of the posets $\underline{2} + \underline{2}$ and $\underline{3} + \underline{1}$.}
\label{Forbidden_posets}
\end{figure}

\begin{theo}[Fishburn Theorem]
A partial order is an interval order if and only if it does not induce any $\underline{2} + \underline{2}$.
\end{theo}

\begin{theo}[Scott-Suppes Theorem]\label{theo_Scott_Suppes}
A partial order is a semiorder if and only if it does not induce any $\underline{2} + \underline{2}$ nor $\underline{3} + \underline{1}$.
\end{theo}

Obviously, any strict weak order is a semiorder, any semiorder is an interval order, and any interval order is a partial order.

We now move on to convex polytopes.  A detailed treatment of the subject can be found for example in \cite{Ziegler_1998}.
A \textsl{convex polytope} in some space $\R^d$ is the convex hull of a finite set of points.  For $X$ a finite subset of $\R^d$, let $P$ be the polytope
$$
P =  \conv \left(\{ x_v \st v \in V \} \right).
$$ 
The \textsl{dimension} $\dim(P)$ of $P$ is the dimension of its affine hull (notice that, except otherwise mentioned, $\dim$ designates the affine dimension). A linear inequality on $\R^d$,
$$
\sum_{i=1}^d \alpha_i x_i \leqslant \beta,
$$
is \textsl{valid} for $P$ if it is satisfied by all points of $P$.  A \textsl{face} of $P$ is the subset of points of $P$ satisfying a given valid inequality with equality (thus $\es$ and $P$ are faces of $P$); then the inequality \textsl{defines} the face. The faces of $P$ are themselves polytopes.  The \textsl{vertices} of $P$ are the points $v$ such that $\{v\}$ is a face of dimension $0$; the \textsl{facets} are the faces of dimension $\dim(P)-1$.  A valid inequality is \textsl{facet-defining} (or a \textsl{FDI}) if it defines a facet of $P$.  The importance of the latter concept is clear from the following result.  Assume that the polytope $P$ is \textsl{full}, that is, of dimension $d$.  Then $P$ equals the set of solutions to all of its facet-defining inequalities; moreover, any system of linear inequalities on $\R^d$ whose set of solutions equals $P$ necessarily contains all of the facet-defining inequalities.

The five polytopes we consider in the paper are defined in the space $\R^{\Arcs{n}}$, where as before $\Arcs{n} = \{(i,j) \st  i,j \in \Vset{n}, i \neq j\}$.
The next lemma recalls, in this setting, a classical result on valid inequalities.  Notice how, for a vector $x \in \R^{\Arcs{n}}$ and $(i,j) \in \Arcs{n}$, we abbreviate $x_{(i,j)}$ into $x_{ij}$. 

\begin{lemm}\label{lem_1} 
Suppose that the two inequalities 
\begin{equation}\label{eq_proper_face}
\sum_{(i,j) \in \Arcs{n}} \alpha_{ij} x_{ij} \leqslant \beta
\end{equation}
and
\begin{equation}\label{eq_proper_face_bis}
\sum_{(i,j) \in \Arcs{n}} \alpha_{ij}' x_{ij} \leqslant \beta'
\end{equation}
define distinct, proper faces of some polytope $P$ in $\R^{\Arcs{n}}$.
Then their sum  
\begin{equation}\label{eq_3}
\sum_{(i,j) \in \Arcs{n}} (\alpha_{ij} + \alpha_{ij}') x_{ij} \leqslant \beta + \beta'
\end{equation}
cannot be facet-defining for $P$.

\end{lemm}

A relation $\RR$ on $\Vset{n}$ is represented in $\R^{\Arcs{n}}$ by its \textsl{characteristic vector} $\chi^R$, with $\chi^{\RR}_{ij} = 1$ if $i \RR j$, and $\chi^R_{ij} = 0$ otherwise.  
The \textsl{semiorder polytope} $\pso n$ on $\Vset{n}$ is defined in 
$\R^{\Arcs{n}}$ by
$$
\pso n =  \conv \left(\left\{ \chi^{\RR}  \st \RR \text{ is a semiorder on } \Vset{n} \right\} \right).
$$
The definitions of the \textsl{linear ordering polytope} $\plo n$,  \textsl{strict weak order polytope} $\pswo n$, \textsl{interval order polytope} $\pio n$, and \textsl{partial order polytope} $\ppo n$  are similar.  Of course 
$$
\plo n \subseteq \pswo n \subseteq \pso n \subseteq \pio n \subseteq \ppo n,
$$
all inclusions being strict when $n \geqslant 4$.
It is easy to see that 
\begin{equation}\label{eqn_dim}
 \dim(\pswo n) = \dim(\pso n) = \dim(\pio n) = \dim(\ppo n) = n(n-1)
\end{equation}
(so that the four polytopes are full), and it can be shown \citep[see for instance][]{Grotschel_Junger_Reinelt_1985a} that 
\begin{equation}\label{eqn_dim_plo}
\dim(\plo n) = \frac{n(n-1)}{2}.
\end{equation}

Remember from our Introduction that characterizing the binary choice probabilities predicted by five models surveyed in \cite{Regenwetter_Davis_Stober_2008, Regenwetter_Davis_Stober_2011, Marley_Regenwetter_2016} amount to listing all of the facet-defining inequalities of the five polytopes $\ppo n$, $\pio n$, $\pso n$, $\pswo n$, and $\plo n$.  The next section explains why the task appears to be much difficult.

\section{The difficulties in finding all FDI's}
\label{sect_diff}

We briefly indicate why finding out all the FDI's of any of the five order polytopes may look as a hopeless task.  The (traditional) trick is to convert some NP-hard combinatorial problem into the optimization problem of a linear form on the polytope.  Among the many possible optimization problems we could select, there are five similar problems which apply in the same way to our five polytopes.  All problems require the construction of a `median order' for a given collection of relations.  We refer the reader to \cite{Hudry_2004,Hudry_2008,Hudry_2012,Hudry_2015} and their references for background on median orders, and also for the results we apply here.

For any family $\RRR$ of relations on $\Vset{n}$ and any relation $P$ on $\Vset{n}$, the \textsl{$\PPP$-remoteness} of $P$ to $\RRR$ equals
\begin{equation}\label{eqn_remoteness}
\rho(P,\RRR) \;=\; \sum_{R \in \RRR} |P \Delta R|,
\end{equation}
where $P \Delta R = (P\setminus R) \cup (R \setminus P)$ is the \textsl{symmetric difference} of $P$ and $R$.  Thus $\rho(P,\RRR)$ counts the total number of disagreements between $P$ and the relations in $\RRR$.
In the next problem, the aim is to minimize the remoteness of an order relation of fixed type to a given family $\RRR$ of relations.  Let     $\PPP$ designate one of the families of linear orders, strict weak orders, semiorders, interval orders of partial orders on $\Vset n$.
 
\begin{prob}[\texttt{$\PPP$-Median Order}]
Given a family $\RRR$ of relations on $\Vset{n}$, find a relation $P$ in $\PPP$ such that $\rho(P,\RRR) \leqslant \rho(Q,\RRR)$ for all $Q$ in $\PPP$.
\end{prob}
Such an order $P$ is \textsl{$\PPP$-median} for the collection $\RRR$ (the notion of a median order plays a role in the aggregation of preferences and in voting theory).

To restate the problem, denote by $\pany{n}$ the corresponding polytope (that is, $\plo n$, $\pwo n$, $\pso n$, $\pio n$ or $\ppo n$) and define from $\RRR$ a vector $c$ in $\R^{\Arcs{n}}$ by
\begin{equation}
c_{ij} \;=\; |\{R\in\RRR \st (i,j)\in R\}| - |\{Q\in\RRR \st (i,j)\notin Q\}|.
\end{equation}
Then 
\begin{equation}\label{eqn_rho}
\rho(P,\RRR) \quad=\quad \sum_{ij\in\Arcs{n}} c_{ij} \chi^P_{ij}
\; - \;
\sum_{R\in\RRR} |R|.
\end{equation}
Because the last sum in Equation~\eqref{eqn_rho} gives a constant depending only on $\RRR$, we see that Problem~\texttt{$\PPP$-Median Order} is equivalent to the minimization of the linear form $x \mapsto c\,x$ (the scalar product of $c$ and $x$) on the polytope $\pany{n}$ (such a reformulation of the problem is well known).

Now it happens that Problem~\texttt{$\PPP$-Median Order} is NP-hard for each choice of $\PPP$ as one of our five collections of order relations---it is even the case under rather strong restrictions on $\RRR$, for instance about the size of $\RRR$ or the type of relations in $\RRR$; \cite{Hudry_2004,Hudry_2015} and \cite{Hudry_Monjardet_2010} provide a wealth of results in this line.  It suffices here to record that Problem~\texttt{$\PPP$-Median Order} is NP-hard and moreover reformulable as a linear programming problem having $\pany{n}$ as its feasible set.  
The following conclusion follows: if a polynomial-size description of $\pany{n}$ existed, the equality $\texttt{P}= \texttt{NP}$ would follow (answering in a surprising way a famous question in complexity theory, \citealp{Garey_Johnson_1979}).  Hence, any linear decription of $\pany{n}$ must be intractable in the technical sense (that is, be of non-polynomial size in $n$).  Of course, the conclusion we just reached does not preclude the existence of an exponential-size, linear description of $\pany{n}$, even one with a nice mathematical structure.

\begin{table}[h]
\begin{equation*}\begin{array}{|c||r|*{6}{r|}}
\hline
n & \multicolumn{2}{c|}{\ppo n} & \multicolumn{2}{c|}{\pio n} & \multicolumn{2}{c|}{\pso n} \\
\hline\hline
2&      3 &  2 &   3 &  2 &   3 &  2 \\
3&     17 &  4 &  17 &  4 &  17 &  4 \\
4&    128 &  8 & 191 & 14 & 563 & 31 \\
5& \geqslant43244 & \geqslant 211  &  &&& \\
\hline
\end{array}
\end{equation*}
\caption{Numbers of FDI's, in total or up to element relabellings, for three order polytopes when, to our knowledge, they are available.}
\label{tab_small_n}
\end{table}

For `small' values of $n$, computers can produce a linear description of $\pany{n}$ (running for instance the software \texttt{porta} \citealp{Christof_porta}).  Tables~\ref{tab_small_n} and \ref{tab_small_n_bis} indicate for which values of $n$, to our knowledge, a description was published.
The cells record, respectively, the total number of FDI's, and the number of their equivalence classes under relabellings of the elements in $\Vset n$.  The values come from
\begin{enumerate}[\quad-~]
\item \cite{Fiorini_thesis} for the partial order polytope $\ppo n$;
\item \cite{Regenwetter_Davis_Stober_2011} for the interval order polytope $\pio n$ and for the semiorder polytope $\pso n$; 
\item \cite{Fiorini_thesis}, \cite{Fiorini_Fishburn_2004} and \cite{Regenwetter_Davis_Stober_2012} for the (strict) weak order polytope $\pswo n$;
\item \cite{Christof_Reinelt_1996} and \cite{Marti_Reinelt_2011} for  the linear ordering polytope $\plo n$.
\end{enumerate}
The symbol ``$\geqslant$'' indicates that the number provided is only a lower bound, the exact value being unknown to us (in many cases, \texttt{porta} was reported to run out of computer resources).

\begin{table}[h]
\begin{equation*}
\begin{array}{|c||r|*{4}{r|}}
\hline
n & \multicolumn{2}{c|}{\pswo n} & \multicolumn{2}{c|}{\quad\plo n}\\[2mm]
\hline\hline
2&   3 &  2 &      2 & 1\\
3&  15 &  3 &      8 & 2\\
4& 106 &  9 &     20 & 2\\
5& 75\,843 & \geqslant318  &  40 & 2\\
6&     &    &    910 & 5\\
7&     &    & 87\,472 & 27\\
8&     &    & \geqslant488\,602\,996 & \geqslant12\,231\\
\hline
\end{array}
\end{equation*}
\caption{Numbers of FDI's, in total or up to element relabellings, for the strict weak order polytope and the linear ordering order polytope when, to our knowledge, they are available.}
\label{tab_small_n_bis}
\end{table}

The difficulty of finding a complete linear description of any of the five polytopes led us to investigate their FDI's under some restriction on the coefficients.  This is why we focus from now on  primary linear inequalities.

\section{General Conditions on the Validity of  Primary Linear Inequalities}\label{gen_fact}

A primary linear inequality on $\R^{\Arcs{n}}$ takes the form
\begin{equation}\label{eqn_rewrite_A_B}
\sum_{a \in A} x_{a} - \sum_{b \in B}x_b \leqslant \beta
\end{equation}
where $\beta\in\{-1,0,1\}$ and $A$, $B$ are two disjoint subsets of $\Arcs{n}$.  From now on, we assume that at least one of $A$ and $B$ is nonempty when we consider Equation~\eqref{eqn_rewrite_A_B}.  Throughout the whole paper, the notation $A$, $B$ and $\beta$ asssume the latter condition (additional assumptions come at the end of the section).  
It is useful to think of $(\Vset{n}, A)$, $(\Vset{n}, B)$ and $(\Vset{n}, A \cup B)$ as graphs\footnote{If not explicitly said otherwise, any graph here is directed, without loops or multiple arcs.}.  When convenient, we will represent the graphs (or parts of them) using solid arcs for the pairs in $A$ and dashed arcs for the pairs in $B$.
Figure~\ref{fig_four_ineq} displays such graphical representations for the five inequalities in Proposition~\ref{prop_five_inequalities}, which provide examples of FDI's for our polytopes. 

\begin{figure}
\begin{center}
\begin{tikzpicture}[scale=1]
  \tikzstyle{vertex}=[circle,draw,fill=white, scale=0.3]
\begin{scope}
 \node (1) at (0,0) [vertex,label=below:$i$] {};
 \node (2) at (0,1) [vertex,label=above:$j$] {};
  \draw[->-=.7,dashed] (1) to (2);
 \node at (0,-1) {$\leqslant 0$};
\end{scope}
  
\begin{scope}[xshift=2.5cm]  
  \node (2-1) at (0,0) [vertex,label=below:$i$] {};
  \node (2-2) at (0,1) [vertex,label=above:$j$] {};
 \draw[->-=.7,>= triangle 45] (2-1) to [bend right=25] (2-2);
 \draw[->-=.7,>= triangle 45] (2-2) to [bend right=25] (2-1);
   \node at (0,-1) {$\leqslant 1$};
\end{scope}

\begin{scope}[xshift=5cm,yshift=-3mm]
  \node (3-i) at (0,0) [vertex,label=below:$i$] {};
  \node (3-j) at (-0.9,1) [vertex,label=left:$j$] {};
  \node (3-k) at (0,2) [vertex,label=above:$k$] {};
 \draw[->-=.7,>= open triangle 45] (3-i) to (3-j);
 \draw[->-=.7,>= open triangle 45] (3-j) to (3-k);
 \draw[->-=.7,>= triangle 45,dashed] (3-i) to (3-k);
   \node at (-0.5,-1) {$\leqslant 1$};
\end{scope}

\begin{scope}[xshift=7.5cm,yshift=-3mm]
  \node (4-i) at (1,0) [vertex,label=below:$i$] {};
  \node (4-j) at (0,2) [vertex,label=above:$j$] {};
  \node (4-k) at (-1,0) [vertex,label=below:$k$] {};
  
 \draw[->-=.7,>= open triangle 45,bend right=15] (4-i) to (4-j);
 \draw[->-=.7,>= open triangle 45,bend right=15] (4-j) to (4-k);
 \draw[->-=.7,>= triangle 45,bend right=15] (4-k) to (4-i);
 
\draw[->-=.7,>= open triangle 45,dashed,bend right=15] (4-i) to (4-k);
\draw[->-=.7,>= open triangle 45,dashed,bend right=15] (4-k) to (4-j);
\draw[->-=.7,>= triangle 45,dashed,bend right=15] (4-j) to (4-i);
   \node at (0,-1) {$\leqslant 1$};
\end{scope}

\begin{scope}[xshift=10cm,yshift=-3mm]
  \node (5-i) at (0,0) [vertex,label=below:$i$] {};
  \node (5-j) at (0,2) [vertex,label=above:$j$] {};
  \node (5-k) at (2,0) [vertex,label=below:$k$] {};
  \node (5-l) at (2,2) [vertex,label=above:$l$] {};
  
 \draw[->-=.7,>= open triangle 45] (5-i) to (5-j);
 \draw[->-=.7,>= open triangle 45] (5-k) to (5-l);
 
\draw[->-=.7,>= open triangle 45,dashed] (5-i) to (5-l);
\draw[->-=.7,>= open triangle 45,dashed] (5-k) to (5-j);
   \node at (1,-1) {$\leqslant 1$};
\end{scope}
    
\end{tikzpicture}
\end{center}
\caption{Graphical representations of the five inequalities in Proposition~\ref{prop_five_inequalities}. The arcs in $A$ are solid, those in $B$ dashed.}
\label{fig_four_ineq}
\end{figure}
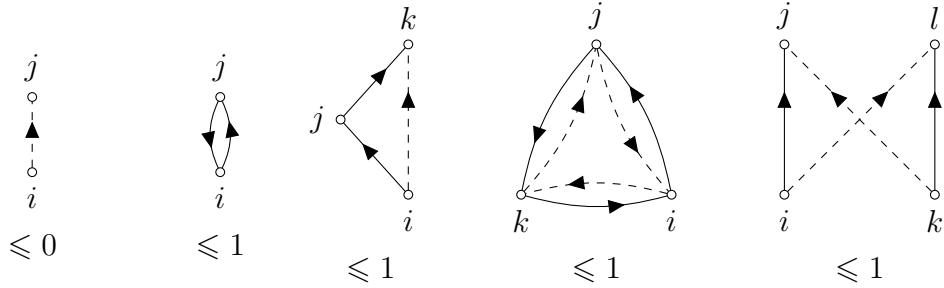 

\begin{prop}\label{prop_five_inequalities}
Consider the five primary linear inequalities
\begin{align}
-x_{ij} \;\leqslant&\; 0,\label{eqn_nonneg}
\\ x_{ij} + x_{ji} \;\leqslant&\; 1,\label{eqn_2}
\\ x_{ij} + x_{jk} - x_{ik} \;\leqslant&\; 1,\label{eqn_3}
\\ x_{ij} + x_{jk} + x_{ki} - x_{ji} - x_{kj} - x_{ik} \;\leqslant&\; 1;\label{eqn_4}
\\x_{ij} + x_{kl} - x_{il} - x_{kj} \;\leqslant&\; 1,\label{eqn_5}
\end{align}
where $i$, $j$, $k$ and $l$ are distinct elements in $\Vset{n}$.  
Table~\ref{table_facets} indicates for which order polytopes they define facets. 
\begin{table}[ht]
\renewcommand{\arraystretch}{1.25}
\caption{Checkmarks indicate when the five linear inequalities define facets of the five order polytopes.\label{table_facets}} 
\begin{center}\begin{tabular}{
c@{\quad}|@{\quad}
c@{\quad}|@{\quad}
c@{\quad}|@{\quad}
c@{\quad}|@{\quad}
c@{\quad}|@{\quad}
c
}
Equation & $\ppo n$ & $\pio n$ & $\pso n$ & $\pswo n$ & $\plo n$ \\
\hline
\eqref{eqn_nonneg} & $\checkmark$ & $\checkmark$ & $\checkmark$ & $\checkmark$ & $\checkmark$ \\
\eqref{eqn_2} & $\checkmark$ & $\checkmark$ & $\checkmark$ & $\checkmark$ &     \\
\eqref{eqn_3} & $\checkmark$ & $\checkmark$ & $\checkmark$ & $\checkmark$ & $\checkmark$ \\
\eqref{eqn_4} & $\checkmark$ & $\checkmark$ & $\checkmark$ & & $\checkmark$ \\
\eqref{eqn_5} &     & $\checkmark$ & $\checkmark$ & & $\checkmark$\\
\end{tabular}
\end{center}

\end{table}
\end{prop}

\begin{proof}
\cite{Suck_1995} (unpublished) establishes many of the results---for $\pso n$, we repeat his argument in Section~\ref{sect_Lifting_Lemma}.  
For $\ppo n$, $\pio n$ and $\plo n$, proofs appear in respectively \cite{Muller_1996}, \cite{Muller_Schulz_1995} and \cite{Fiorini_thesis}.
For $\pso n$ the results are consequences of findings of \cite{Regenwetter_Davis_Stober_2011} and Theorem~\ref{theo_lifting} below.
For $\pswo n$, the assertions follow from results in \cite{Fiorini_Fishburn_2004} on the weak order polytope (the relationship between the latter polytope and the strict weak order polytope is described in Section~\ref{sect_Primary_FDI_PWO}).  For $\plo n$, see for instance \cite{Grotschel_Junger_Reinelt_1985a}.
\end{proof}

We call \textsl{nonnegativity inequality} the inequality in Equation~\eqref{eqn_nonneg}.

Our main contribution is to characterize the primary linear inequalities that define facets of the polytopes $\ppo n$, $\pio n$ and $\pso n$.
For the rest of this section, let $P_n$ be one of the latter three polytopes (we consider $\pswo n$ and $\plo n$ in Sections~\ref{sect_Primary_FDI_PWO} and \ref{sect_Primary_FDI_PLO}, respectively).  
The following four conditions on the sets $A$ and $B$ are most useful in our characterization of valid or facet-defining, primary inequalities for $P_n$ in the form of Inequality~\eqref{eqn_rewrite_A_B}.  Figure~\ref{fig_C2C3C4} illustrates Conditions~C2, C3 and C4.\\

\noindent \textbf{Condition~C1}: For any distinct pairs $(i,j)$, $(k,l)$ in $A$, there holds $i \neq k$ and $j \neq l$.\\

\noindent \textbf{Condition~C2}: For any $(i,j),$ $(k,l)$ in $A$ with $i$, $j$, $k$, $l$ distinct elements, there holds $(i,l) \in B$.\\

\noindent \textbf{Condition~C3}:  For any pairs $(i,j)$, $(j,k)$ in $A$ with $i$, $j$, $k$ distinct elements, there holds $(i,k)\in B$.\\

\noindent  \textbf{Condition~C4}:  For any pairs $(i,j)$, $(j,k)$ in $A$ with $i$, $j$, $k$ distinct elements, either there holds $(i,k) \in B$ or there exists some element $p$ in $\Vset{n}\setminus\{i,j,k\}$ such that $(i,p),(p,k)\in B$.\\

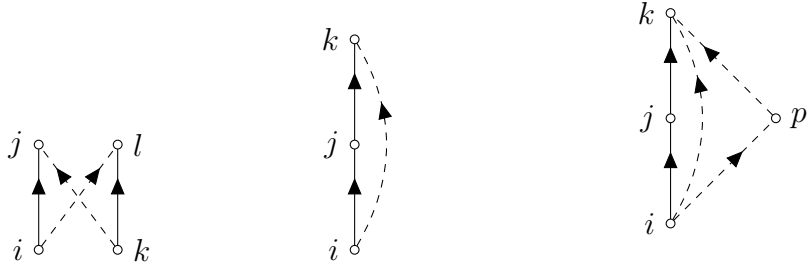
\begin{figure}
\begin{center}
\begin{tikzpicture}[scale=0.7]
  \tikzstyle{vertex}=[circle,draw,fill=white, scale=0.3]
\begin{scope}
  \node (z_l) at (1.5,2) [vertex] {};
  \node (z_k) at (1.5,0) [vertex] {};
  \node (z_j) at (0,2) [vertex] {};
  \node (z_i) at (0,0) [vertex] {};
  
  \draw[->-=.7] (z_i) to (z_j);
  \draw[->-=.7] (z_k) to (z_l);
  
   \draw[->-=.8, dashed] (z_i) to (z_l);
   \draw[->-=.8, dashed] (z_k) to (z_j);
   
   \draw (z_i.west) node [left] {$i$};
   \draw (z_j.west) node [left] {$j$};
   \draw (z_k.east) node [right] {$k$};
   \draw (z_l.east) node [right] {$l$};
\end{scope}
  

\begin{scope}[xshift=6cm]  
  \node (u_k) at (0,4) [vertex] {};
  \node (u_j) at (0,2) [vertex] {};
  \node (u_i) at (0,0) [vertex] {};
  
  \draw[->-=.7] (u_i) to (u_j);
  \draw[->-=.7,] (u_j) to (u_k);
  \draw[->-=.7, dashed] (u_i) to [bend right] (u_k);
  
   \draw (u_i.west) node [left] {$i$};
   \draw (u_j.west) node [left] {$j$};
    \draw (u_k.west) node [left] {$k$};
\end{scope}

\begin{scope}[xshift=12cm,yshift=0.5cm]
  \node (i) at (0,0) [vertex,label=left:$i$] {};
  \node (j) at (0,2) [vertex,label=left:{$j$}] {};
  \node (k) at (0,4) [vertex,label=left:$k$] {};
  \draw[->-=.7] (i) -- (j);
  \draw[->-=.7] (j) -- (k);
  \draw[->-=.7, dashed] (i) to [bend right] (k);

  \node (p) at (2,2) [vertex,label=right:$p$] {};
  \draw[->-=.7, dashed] (i) -- (p);
  \draw[->-=.7, dashed] (p) -- (k);
\end{scope}
    
\end{tikzpicture}
\end{center}
\caption{Illustrations of Conditions~C2, C3 and C4 on subsets $A$, $B$ of $\Arcs{n}$. The solid arcs are in $A$, the dashed ones in $B$.}
\label{fig_C2C3C4}
\end{figure}

Note that Condition~C1 exactly means that the graph is a disjoint union of isolated vertices, directed paths and directed cycles (where ``disjoint union'' means ``no two constituents have a vertex in common'').  We call \textsl{PC-graph} any graph of this type.  In Condition~C2, $(k,j) \in B$ also follows from the assumption (by a renaming of the elements).

\begin{theo}\label{thm_valid_P}
If Inequality~\eqref{eqn_rewrite_A_B} is valid for $P_n$, we have:
\begin{enumerate}[\quad\rm 1.~]
\item $\beta \geqslant 0$; if $A\neq\es$, then $\beta = 1$;
\item the set $A$ satisfies Condition~C1;
\item the sets $A$ and $B$ satisfy Condition~C2;
\item if $P_n = \ppo n$ or $P_n = \pio n$, then $A$ and $B$ satisfy Condition~C3;
\item if $P_n = \pso n$, then $A$ and $B$ satisfy Condition~C4.
\end{enumerate}
\end{theo}

\begin{proof}1)~The incidence vector of the empty relation (or antichain) on $\Vset{n}$ gives value $0$ to the left-hand side of Inequality~\eqref{eqn_rewrite_A_B}.  Validity of Inequality~\eqref{eqn_rewrite_A_B} then implies $\beta \geqslant 0$.
Next, let $a$ be in $A$.  Consider the semiorder consisting only of the pair $\{a\}$; its characteristic vector gives value $1$ to the left-hand side of Inequality~\eqref{eqn_rewrite_A_B}.  Validity of the inequality thus implies $\beta \geqslant 1$.  Moreover, we assume $\beta\in\{-1$, $0$, $1\}$.

2)~To show that Condition~C1 is satisfied, suppose first the existence of pairwise distinct elements $i$, $j$, $k$ in $\Vset{n}$ with $(i,j)$, $(i,k)$ in $A$.  Then the vector of $\R^{A_n}$ defined by
\begin{equation*}
    y_{uv} =
    \begin{cases}
      1 & \text{if } (u,v)=(i,j) \text{ or } (u,v)=(i,k),\\
      0 & \text{otherwise }
    \end{cases}
\end{equation*}
is the incidence vector of a semiorder, and moreover 
$$
\sum_{a \in A} y_{a} - \sum_{b \in B}y_b  = 2 > \beta.
$$
This contradicts the validity of Inequality~\eqref{eqn_rewrite_A_B}. 
The argument is similar in case $(i,k), (j,k) \in A$.

3)~To establish Condition~C2, consider $(i,j),$ $(k,l) \in A$ with $i$, $j$, $k$, $l$ pairwise distinct elements. Suppose $(i,l) \notin B$. Then the vector of $\R^{A_n}$ defined by
\begin{equation*}
    y_{uv} =
    \begin{cases}
      1 & \text{if } (u,v)\in \{(i,j),\,(k,l),\,(i,l)\},\\
      0 & \text{otherwise}
    \end{cases}
\end{equation*}
is the incidence vector of a semiorder, and 
$$
\sum_{a \in A} y_{a} - \sum_{b \in B}y_b  = 2 > \beta.
$$

4)~Let $(i,j), (j,k) \in A$ and suppose $(i,k) \notin B$. Then the vector of $\R^{A_n}$ defined by
\begin{equation*}
    y_{uv} =
    \begin{cases}
      1 & \text{if } (u,v) \in \{(i,j),\,(i,k),\,(j,k)\},\\
      0 & \text{otherwise }
    \end{cases}
\end{equation*}
is the incidence vector of an interval order such that 
$$
\sum_{a \in A} y_{a} - \sum_{b \in B}y_b  \in \{2,\,3\},
$$
a contradiction with the validity of Inequality~\eqref{eqn_rewrite_A_B} for $\pio n$.  Note that the interval order we just constructed is in general not a semiorder.

5) Finally, suppose Inequality~\eqref{eqn_rewrite_A_B} is valid for $\pso n$ but does not satisfy Condition~C4.  Let the pairs $(i,j)$, $(j,k)$ in $A$ invalidate Condition~C4. Then $(i,k) \notin B$.  Moreover, for evey element $p$ in $\Vset{n}\setminus\{i,j,k\}$ we have $(i,p) \notin B$ or $(p,k)\notin B$; we call $z_p$ the pair not in $B$ (if none of the two pairs happens to be in $B$, we choose $z_p$ arbitrary among them).  Then 
$$
\{(i,j), (i,k), (j,k), z_p \st p \in \Vset{n}\setminus\{i,j,k\}\}
$$
is a semiorder on $\Vset{n}$ whose characteristic vector invalidates Inequality~\eqref{eqn_rewrite_A_B}, a contradiction. 
\end{proof}

Additional conditions obtain when Inequality~\eqref{eqn_rewrite_A_B} is moreover facet-defining.
 
\begin{theo}\label{thm_FDI_P}
Suppose that Inequality~\eqref{eqn_rewrite_A_B} defines a facet of $P_n$. Then:
\begin{enumerate}[\quad\rm 1.~]
\item if $\beta = 0$, then $A$ is empty and $B$ is a singleton, in other words Inequality~\eqref{eqn_rewrite_A_B} is a nonnegativity equality~\eqref{eqn_nonneg};
\item if $\beta = 1$, the set $A$ can neither be empty nor be a singleton.
\end{enumerate}
\end{theo}
	
\begin{proof}1)~If $\beta=0$, we must have $A = \varnothing$ (Theorem~\ref{thm_valid_P}).  Now the inequality reads $-\sum_{b \in B}x_b \leqslant 0$, with $B$ nonempty.  If $|B|=1$, we get a nonnegativity inequality.  If $|B| \geqslant 2$, Lemma \ref{lem_1} shows that our inequality cannot define a facet. 
	
2)~If $A=\es$, equality cannot be reached in \eqref{eqn_rewrite_A_B}.  Suppose now $A = \{(i,j)\}$. Then if $y$ is a vertex of $P_n$ satisfying equality in Inequality~\eqref{eqn_rewrite_A_B}, we have $y_{ij} = 1$ and hence $y_{ji} = 0$.  So $y$ is contained in the facet defined by $-x_{ij} \leqslant 0$, a contradiction.
\end{proof}
 
Because of Theorem~\ref{thm_FDI_P}, when searching for primary FDI's of $P_n$ in the form of Equation~\eqref{eqn_rewrite_A_B}, we will always suppose $\beta = 1$, $|A| \geqslant 2$ and $B \neq \varnothing$: indeed, the only FDI not satisfying these requirements is $-x_{ij} \leqslant 0$.

In the next two sections, we determine explicitly all the primary FDI's of respectively $\ppo n$ and $\pio n$.

\section{The Primary Facet-Defining Inequalities of the Partial Order Polytope}
\label{sect_Primary_FDI_PPO}

Among other authors, \cite{Muller_1996} and, later, \cite{Fiorini_2003} investigate mathematical aspects of the partial order polytope.  It happens that their (unsurprisingly) incomplete lists of FDI's nevertheless contain all the primary ones---as we show in Theorem~\ref{main_2} below.  The partial order polytope appears in psychological applications, for instance in \cite{Regenwetter_Davis_Stober_2008}.    

\begin{theo}\label{main_1}
The inequality
\begin{equation}\label{eqn_rerewrite_A_B}
\sum_{a \in A} x_{a} - \sum_{b \in B}x_b \leqslant 1
\end{equation}
is valid for the partial order polytope $\ppo n$ if and only if the pair $A$, $B$ satisfies one of the following requirements, for some distinct elements $i$,$j$ and $k$:
\begin{enumerate}[\quad\rm 1.]
\item $A=\{(i,j)\}$;
\item $A=\{(i,j),\,(j,i)\}$;
\item $A=\{(i,j),\,(j,k)\}$ and $(i,k)\in B$;
\item $A=\{(i,j),\,(j,k),\, (k,i)\}$ and $(j,i),\, (k,j),\, (i,k)\in B$.
\end{enumerate}
\end{theo}

\begin{proof}
We first show that if Inequality~\eqref{eqn_rerewrite_A_B} is valid for $\ppo n$, then $A$ satisfies\\

\noindent{Condition~C0}: For any distinct pairs $(i,j)$, $(k,l)$ in $A$, the elements $i$, $j$, $k$ and $l$ are not all distinct.\\[2mm]
Indeed, if the elements $i$, $j$, $k$ and $l$ were distinct the characteristic vector of the partial order $\{(i,j)$, $(k,l)\}$ would invalidate \eqref{eqn_rerewrite_A_B}.  

By Theorem~\ref{thm_valid_P}, Condition~C1 also holds.  The only subsets $A$ of $\Arcs n$ that satisfy both Conditions~C0 and C1 are exactly those described in the requirements of the theorem.  Condition~C3 (which holds by Theorem~\ref{thm_valid_P}.4) then implies the rest of the requirements.

Conversely, any of the inequalities satisfying one of the requirements is the sum of one of the FDI's \eqref{eqn_nonneg}--\eqref{eqn_4} with a valid inequality of the form $\sum_{c\in C} -x_c \le 0$ for some $C \subseteq \Arcs{n} $, so it is also valid.
\end{proof}

We now describe the FDI's of $\ppo n$.

\begin{theo}\label{main_2}
The primary FDI's of the partial order polytope $\ppo n$ are exactly those in Equations~\eqref{eqn_PPO_1}, \eqref{eqn_PPO_2}, \eqref{eqn_PPO_3} and \eqref{eqn_PPO_4},
that is:
\begin{align}
-x_{ij} \leqslant&\; 0,\label{eqn_PPO_1}
\\ x_{ij} + x_{ji} \leqslant&\; 1, \label{eqn_PPO_2}
\\ x_{ij} + x_{jk} - x_{ik} \leqslant&\; 1,\label{eqn_PPO_3}
\\ x_{ij} + x_{jk} + x_{ki} - x_{ji} - x_{kj} - x_{ik} \leqslant&\; 1\label{eqn_PPO_4}
\end{align}
(where $i,j,k$ are pairwise distinct elements of $\Vset{n}$).
\end{theo}

\begin{proof} As mentioned in Proposition~\ref{prop_five_inequalities}, Inequalities \eqref{eqn_PPO_1}--\eqref{eqn_PPO_4} define facets of $\ppo n$.  Among the valid inequalities described in Theorem~\ref{main_1}, only inequalities \eqref{eqn_PPO_1}--\eqref{eqn_PPO_4} are FDI's.  This follows by the application of Lemma~\ref{lem_1} to the latter inequalities and inequalities of the form $\sum_{c\in C} -x_c \le 0$ for well-chosen subsets $C$ of $\Arcs{n}$.
\end{proof}

\section{The Primary Facet-Defining Inequalities of the Interval Order Polytope}
\label{sect_Primary_FDI_PIO}

We now turn to the interval order polytope $\pio n$ (\citealp{Muller_Schulz_1995}; \citealp{Suck_1995}; \citealp{Regenwetter_Davis_Stober_2008}; \citealp{Regenwetter_Davis_Stober_2011}, for instance).  

\begin{theo}\label{thm_valid_pio}
The primary linear inequality
\begin{equation}\label{eqn_02_schtroumpf}
\sum_{a \in A} x_{a} - \sum_{b \in B}x_b \leqslant 1
\end{equation}
is valid for $\pio n$ if and only if $A$ and $B$ satisfy Conditions~C1, C2 and C3.  
\end{theo}

Thus Inequality~\eqref{eqn_02_schtroumpf} is valid for $\pio n$ exactly if $(\Vset n,A)$ is a PC-graph and moreover $B$ contains all pairs required to make Conditions~C2 and C3 true---and maybe additional pairs.

\begin{proof}
The necessity of Conditions~C1--C3 was established in Proposition~\ref{thm_valid_P}.4.

The sufficiency of Conditions~C1--C3 obtains as follows.  Let $P$ be an interval order on $\Vset n$ with $(i,j)\in P \cap A$ (if $P \cap A=\es$ validity is obvious).  Now for each other pair $(k,l)$ in $P \cap A$, Condition~C1 implies $i\neq k$ and $j\neq l$.  If moreover $i \neq l$ and $j\neq k$, then Condition~C2 implies that $B$ contains both pairs $(i,l)$ and $(k,j)$; at least one of them must be in $P$.  If $i = l$, Condition~C3 implies $(k,j)\in B$; on the other hand, $(k,j)\in P$.  The case $j = k$ similarly gives $(i,l)\in B \cap P$.  Notice that distinct pairs $(k,l)$ in $P \cap A$ (all different from the initial $(i,j)$) deliver in this way distinct pairs in $B \cap P$ (because of Condition~C1).  The validity of \eqref{eqn_02_schtroumpf} follows. 
\end{proof}

From Theorems~\ref{thm_FDI_P}.2 and \ref{thm_valid_pio}, we know that if the primary linear inequality 
\begin{equation}\label{eqn_02}
\sum_{a \in A} x_{a} - \sum_{b \in B}x_b \leqslant 1
\end{equation}
defines a facet of $\pio n$, then $|A| \geqslant 2$ and Conditions~C1, C2, C3 hold.  
Let us say that $B$ is $A$[C2--C3]-\textsl{forced} when $B$ consists exactly of the pairs whose existence is required in Conditions~C2 and C3, that is when
\begin{equation}\label{eqn_PIO_in_B}
(i,l) \in B \quad\iff 
\begin{cases}
&\exists (i,j),(k,l)\in A \text{ with } i,j,k,l \text{ distinct, or}\\
&\exists (i,j),(j,l)\in A \text{ with } i,j,l \text{ distinct}.
\end{cases}
\end{equation}
A result of \cite{Muller_Schulz_1995} states that when $A$ and $B$ satisfy $|A| \geqslant 2$, Conditions~C1, C2, C3 and moreover $B$ is $A$[C2--C3]-forced, then Inequality~\eqref{eqn_02} is facet-defining (in a more general context, \citeauthor{Muller_Schulz_1995} call `io-clique inequalities' such inequalities).  We now easily establish  that no further primary FDI exists.

\begin{theo}\label{thm_FDI_pio}
The primary FDI's of $\pio n$ are exactly the nonnegativity inequalities and the inequalities
\begin{equation}\label{eqn_02_bis}
\sum_{a \in A} x_{a} - \sum_{b \in B}x_b \leqslant 1
\end{equation}
for which $|A| \geqslant 2$, Conditions~C1, C2, C3 hold, and moreover  $B$ is $A$[C2--C3]-forced.  
\end{theo}

\begin{proof}
\cite{Muller_Schulz_1995} establish that Inequality~\eqref{eqn_02_bis} is a FDI when the requirements are satisfied.  For the converse, by Theorem~\ref{thm_valid_P}, there remains to prove that, for a primary FDI \eqref{eqn_02} of $\pio n$ other than a nonnegative inegality, all pairs in $B$ are $A$[C2--C3]-forced.  But this follows at once from Lemma~\ref{lem_1}, the result of \citeauthor{Muller_Schulz_1995} and the validity of $-x_{ij}\leqslant 0$. 
\end{proof}

\begin{exam}\label{exam_n_fence}
The $n$-\textsl{fence inequality} is just a particular case of the  inequalities as in Theorem~\ref{thm_FDI_pio}, in which any two pairs in $A$ are disjoint (two pairs $(i,j)$ and $(k,l)$ of elements are \textsl{(vertex) disjoint} when $\{i,j\} \cap \{k,l\}=\es$).
As in Figure~\ref{fig_fence}, let $A = \{(a_1, b_1)$, $(a_2,b_2)$, \dots, $(a_m, b_m)\}$ with $m\geqslant 1$ and $|\{a_1$, $a_2$, \dots, $a_m$, $b_1$, $b_2$, \dots, $b_m\}|=2m$
(thus $2m \leqslant n$).  The resulting set $B$ contains all pairs $(a_i,b_j)$ such that $i \neq j$.  Thus the corresponding primary  linear inequality reads
\begin{equation}\label{eqn_09}
\sum_{i=1}^{m}x_{a_ib_i} - 
\sum_{\substack{i,j=1 \\ i \neq j}}^m x_{a_ib_j} 
\;\leqslant\; 1.
\end{equation}
\begin{figure}
\begin{center}
\begin{tikzpicture}
  \tikzstyle{vertex}=[circle,draw,fill=white, scale=0.5]
 
  \node at (3.5,2) {\textbf{...}};
  \node at (4.5,1.3) {\textbf{...}};  
  \node at (4.5,0.7) {\textbf{...}};
  \node at (3.5,0) {\textbf{...}};
  
  \node (z_1) at (0,0) [vertex] {};
  \node (z_2) at (0,2) [vertex] {};
  \node (z_3) at (2,0) [vertex] {};
  \node (z_4) at (2,2) [vertex] {};
  \node (z_5) at (5,0) [vertex] {};
  \node (z_6) at (5,2) [vertex] {};
  
  \draw[->-=.7] (z_1) to (z_2);
  \draw[->-=.7] (z_3) to (z_4);
  \draw[->-=.7] (z_5) to (z_6);

   \draw (z_1.south) node [below] {$a_1$};
   \draw (z_3.south) node [below] {$a_2$};
   \draw (z_5.south) node [below] {$a_m$};
   
   \draw (z_2.north) node [above] {$b_1$};
   \draw (z_4.north) node [above] {$b_2$};
   \draw (z_6.north) node [above] {$b_m$};
   
    \draw[->-=.7, dashed] (z_1) to (z_4);
    \draw[->-=.7, dashed] (z_1) to (z_6);
    \draw[->-=.7, dashed] (z_3) to (z_2);
    \draw[->-=.7, dashed] (z_3) to (z_6);
    
     \draw[->-=.7, dashed] (z_5) to (z_2);
     \draw[->-=.7, dashed] (z_5) to (z_4);
\end{tikzpicture}
\end{center}
\caption{}\label{fig_fence}
\end{figure}
\end{exam}

The $n$-fence inequality first appeared in studies of the linear ordering polytope (see Section~\ref{sect_Primary_FDI_PLO}).  \cite{Muller_Schulz_1995} mentions that it defines a facet of $\pio n$. 

Here is an immediate corollary of Theorem~\ref{thm_FDI_pio}.

\begin{coro}
There exists a bijection between the primary FDI's of $\pio n$ and the PC graphs on $n$ elements with at least one arc. 
\end{coro}

\begin{proof}
Let 
$$
\sum_{a \in A} x_{a} - \sum_{b \in B}x_b \leqslant \beta
$$
be a primary FDI of $\pio n$.  If $\beta = 0$ then by Theorem~\ref{thm_FDI_P}.1, we have a nonnegativity inequality, that is $A = \varnothing$ and there exist $(i,j) \in \Arcs{n}$ such that $B = \{(i,j)\}$.  To the inequality we associate the PC graph on $\Vset{n}$ having one single arc $\{(i,j)\}$.  If $\beta = 1$ then $A$ contains at least two arcs and to the inequality we associate the PC graph $(\Vset{n},A)$.  This association specifies the desired bijection.  Indeed, surjectivity is obvious, and injectivity holds because $A$ univocally determines $B$.
\end{proof}

\section{The Lifting Lemma for the Semiorder Polytope}
\label{sect_Lifting_Lemma}

We now turn to the semiorder polytope $\pso n$.
The polytope is mentioned for instance in an unpublished manuscript of \citet{Suck_1995} and in papers by \citet{Regenwetter_Davis_Stober_2008, Regenwetter_Davis_Stober_2011}.
The following theorem entails that any facet-defining inequality for $\pso n$ remains facet-defining for $\pso k$ for all $k$ such that $k  \geqslant n$---it is common to name such a statement the \textsl{Lifting Lemma}.  \cite{Suck_1995} states the theorem and gives a sketch of proof.  Notice $\dim\left(\pso {n+1}\right) =  2n + \dim\left(\pso {n}\right)$, because Equation~\eqref{eqn_dim} reads $\dim \pso n = n(n-1)$.

\begin{theo}\label{theo_lifting}
Let 
\begin{equation}\label{eqn_before}
\sum_{(i,j) \in \Arcs{n}} \alpha_{ij}x_{ij} \leqslant \beta 
\end{equation}
be an inequality on $\R^\Arcs{n}$ which is valid for $\pso n$ (where $\alpha_{ij}$, $\beta \in \R$).  Let $F$ be the face of $\pso n$ defined by \eqref{eqn_before}.  Consider then the inequality on $\R^{\Arcs{n+1}}$
\begin{equation}\label{eqn-after}
\sum_{(i,j) \in \Arcs{n+1}} \alpha_{ij}'x_{ij} \leqslant \beta
\end{equation}
with $\alpha_{ij}' = \alpha_{ij}$ if $i,j \in \Vset{n}$ and $\alpha_{ij}' = 0$ if $i = n+1$ or $j = n+1$.  Then Equation~\eqref{eqn-after} is valid for $\pso {n+1}$, and it defines a face $F'$ of $\pso{n+1}$ of dimension 
\begin{equation}
\dim(F') = 2n + \dim(F).
\end{equation}
In particular, if Equation~\eqref{eqn_before} defines a facet of $\pso n$, then Equation~\eqref{eqn-after} defines a facet of $\pso{n+1}$. 
\end{theo}

\begin{proof}  
Consider the canonical linear projection $\pi: \R^{\Arcs{n+1}} \to \R^{\Arcs{n}}$ (whose effect is to delete all coordinates attached to a pair of elements one of which is $n+1$).
Then $\pi$ maps the characteristic vector of a semiorder $S$  on $\Vset{n+1}$ to the characteristic vector of the restriction of $S$ to $\Vset{n}$---and this restriction is a semiorder on $\Vset{n}$.  Hence Equation~\eqref{eqn-after} is valid for $\pso{n+1}$ and it defines a face $F'$ of $\pso{n+1}$. 
From previous sentence there follows $\pi(F') \subseteq F$, and then  $F' \subseteq \pi^{-1}(F)$.
Now because the kernel of the linear mapping $\pi$ has dimension $2n$, we derive
$$  
\dim(F') \leqslant 2n + \dim(F).
$$
To prove the opposite inequality, let $k=\dim(F)$ and select $1+k$ affinely independent vertices $v_0$, $v_1$, \dots, $v_k$ in $F$.  Thus each $v_i$ is the characteristic vector of some semiorder $R_i$ on $\Vset{n}$.  Adding to $R_i$ all pairs $(n+1,i)$ for $i\in \Vset{n}$, we get a semiorder $R_i'$ on $\Vset{n+1}$;
denote by $v'_i$ its characteristic vector in $\R^{\Arcs{n+1}}$.  All $v'_i$'s belong to $F'$.  Moreover the points $v'_0$, $v'_1$, \dots, $v'_k$ are affinely independent because $\pi(v'_i)=v_i$.
To prove $\dim(F') \geqslant 2n + \dim(F)$, it suffices now to build $2n$ vertices $w_1$, $w_2$, \dots, $w_{2n}$ in $F' \cap \pi^{-1}(v_0)$ such that $v'_0$, $w_1$, $w_2$, \dots, $w_{2n}$ are affinely independent (because then $v'_0$, $w_1$, $w_2$, \dots, $w_{2n}$,  $v'_1$, \dots, $v'_k$ are together affinely independent).  

Consider some unit interval representation of $R_0$.  First, by slightly perturbing the real values $f(i)$ if necessary, we make the  $2n$ real values $f(i)-1$ and $f(i)$, for $i=1$, $2$, \dots, $n$, distinct and list them in increasing order as $\gamma_1$, $\gamma_2$, \dots, $\gamma_{2n}$.
We then form $1+2n$ semiorders on $\Vset{n+1}$ by specifying one of their interval representations: we always leave unchanged the values $f(i)$ (that is, the actual intervals $[f(i),f(i)+1]$) for $1 \leqslant i \leqslant n$, but select for $f(n+1)$ various values: first a value strictly below $\gamma_1$, then a value strictly between  $\gamma_1$ and $\gamma_2$,  next a value strictly between $\gamma_2$ and $\gamma_3$, \dots, and finally a value strictly above $\gamma(2n)$.  There results $1+2n$ semiorders $R'_0$, $S_1$, $S_2$, \dots, $S_{2n}$ on $\Vset{n+1}$ whose characteristic vectors we denote by $v'_0$, $w_1$, $w_2$, \dots, $w_{2n}$.  Notice that $\pi$ maps all these vectors to $v_0$.  Moreover the points $v'_0$, $w_1$, $w_2$, \dots, $w_{2n}$ are affinely independent, because either the semiorder $S_j$ lacks a pair $(n+1,i)$ which belongs to all previous semiorders or $S_j$ contains a pair $(i,n+1)$ which was in none of the previous semiorders. 

We have thus proved $\dim(F') = 2n + \dim(F)$.  The very last assertion directly follows.
\end{proof}

Theorem~\ref{theo_lifting} motivates the following:
When we study an inequality of the form $\sum_{a \in A} x_{a} - \sum_{b \in B}x_b \leqslant \beta$ for $\pso n$ , we may restrict ourselves to the smallest $n$ such that $A \cup B \subseteq \Arcs{n}$.  If this inequality is valid (respectively, facet-defining) for $\pso n$ it will also be valid (respectively, facet-defining) for $\pso k$ with $k \geqslant n$.

This is how \cite{Suck_1995} establishes several families of primary FDI's for $\pso n$ (some were already mentioned in Proposition~\ref{prop_five_inequalities}).    

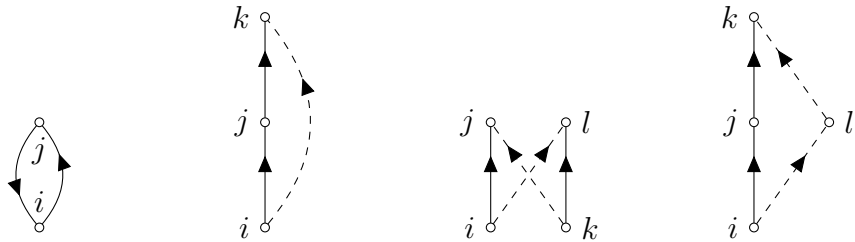
\begin{figure}[h]
\begin{center}
\begin{tikzpicture}[yscale=0.7]
  \tikzstyle{vertex}=[circle,draw,fill=white, scale=0.3]
  
  \node (v_i) at (0,0) [vertex] {};
  \node (v_j) at (0,2) [vertex] {};
  
   \draw (v_i.north) node [above] {$i$};
   \draw (v_j.south) node [below] {$j$};

  \draw[->-=.7] (v_i) to [bend right] (v_j);
  \draw[->-=.7] (v_j) to [bend right] (v_i);
  	    
\begin{scope}[xshift=6cm]
  \node (z_i) at (0,0) [vertex] {};
  \node (z_j) at (0,2) [vertex] {};
  \node (z_k) at (1,0) [vertex] {};
  \node (z_l) at (1,2) [vertex] {};

  \draw[->-=.7] (z_i) to (z_j);
  \draw[->-=.7] (z_k) to (z_l);
  
   \draw[->-=.8, dashed] (z_i) to (z_l);
   \draw[->-=.8, dashed] (z_k) to (z_j);
   
   \draw (z_i.west) node [left] {$i$};
   \draw (z_j.west) node [left] {$j$};
   \draw (z_k.east) node [right] {$k$};
   \draw (z_l.east) node [right] {$l$};
\end{scope}
\begin{scope}[xshift=3cm]
  \node (u_k) at (0,4) [vertex] {};
  \node (u_j) at (0,2) [vertex] {};
  \node (u_i) at (0,0) [vertex] {};
  
  \draw[->-=.7] (u_i) to (u_j);
  \draw[->-=.7] (u_j) to (u_k);
  \draw[->-=.7, dashed] (u_i) to [bend right] (u_k);
  
   \draw (u_i.west) node [left] {$i$};
   \draw (u_j.west) node [left] {$j$};
    \draw (u_k.west) node [left] {$k$};
\end{scope}
\begin{scope}[xshift=9.5cm]
  \node (t_k) at (0,4) [vertex] {};
  \node (t_j) at (0,2) [vertex] {};
  \node (t_i) at (0,0) [vertex] {};
  \node (t_l) at (1,2) [vertex] {};
  
  \draw[->-=.7] (t_i) to (t_j);
  \draw[->-=.7] (t_j) to (t_k);
  \draw[->-=.7, dashed] (t_i) to (t_l);
  \draw[->-=.7, dashed] (t_l) to (t_k);
  
   \draw (t_i.west) node [left] {$i$};
   \draw (t_j.west) node [left] {$j$};
   \draw (t_k.west) node [left] {$k$};
   \draw (t_l.east) node [right] {$l$};
\end{scope}
\end{tikzpicture}
\end{center}
\caption{\label{fig_axiomatic}
Graphical representations of the Axiomatic Inequalities \eqref{eqn_axiomatic1}--\eqref{eqn_axiomatic4} for $\pso n$ (with independent term equal to $+1$).}
\end{figure}

For $i$, $j$, $k$, $l$ any distinct elements of $\Vset{n}$, the four inequalities (see also Figure~\ref{fig_axiomatic})
\begin{align}
x_{ij} + x_{ji}  &\leqslant 1,\label{eqn_axiomatic1}\\
x_{ij} + x_{jk} - x_{ik} &\leqslant 1,\label{eqn_axiomatic2}\\
x_{ij} + x_{kl} - x_{il} - x_{kj} &\leqslant 1,\label{eqn_axiomatic3}\\
x_{ij} + x_{jk} - x_{il} - x_{lk} &\leqslant 1\label{eqn_axiomatic4}
\end{align}
are valid for $\pso n$.  
Because Equations~\eqref{eqn_axiomatic1}--\eqref{eqn_axiomatic4} derive from the conditions in the (Scott-Suppes) Theorem~\ref{theo_Scott_Suppes} (for the first and second ones, remember that neither $(i,i)$ nor $(j,j)$ appears in any semiorder), we call them the \textsl{Axiomatic Inequalities}.  
  
\begin{theo}[Axiomatic FDIs of $\pso n$, \citealp{Suck_1995}]\label{theo_basic_axiomatic}
For $i,j,k,l \in \Vset{n}$ pairwise distinct, the Axiomatic Inequalities~\eqref{eqn_axiomatic1}--\eqref{eqn_axiomatic4} define facets of $\pso n$.  
\end{theo}

\section{The Primary Valid Inequalities of the Semiorder Polytope}\label{sec_valid_PSO}

Theorem~\ref{thm_FDI_P} states that if the primary linear inequality
\begin{equation}\label{eqn_03}
\sum_{a \in A} x_{a} - \sum_{b \in B}x_b \leqslant 1
\end{equation}
is valid for the semiorder polytope $\pso n$, then $A$ and $B$ must satisfy Conditions~C1, C2 and C4.  As we will see in the next proof, the converse does not hold and we now proceed to define an additional necessary condition.\\

\noindent \textbf{Condition~C5}: For any pairs $(i,j)$, $(j,k)$ and $(k,l)$ in $A$ with $i$, $j$, $k$ and $l$ distinct elements, at least one the following requirements holds: 
\begin{enumerate}[\quad({v}1)]
\item $(i,k)\in B$;
\item $(j,l)\in B$;
\item there exists some $r$ in $\Vset{n}\setminus\{i,j,k,l\}$ such that $(i,r),(r,k)\in B$;
\item there exists some $s$ in $\Vset{n}\setminus\{i,j,k,l\}$ such that $(j,s),(s,l)\in B$;
\item there exists some $t$ in $\Vset{n}\setminus\{i,j,k,l\}$ such that $(i,t),(t,l)\in B$;
\item there exist some $u$ and $v$ in $\Vset{n}\setminus\{i,j,k,l\}$ such that $(i,u),(u,v)$, $(v,l)\in B$.
\end{enumerate}

An illustration is given in Figure~\ref{fig_C5}.  Taking into considerations $i$, $j$ and $k$ in Conditions~C4 and C5 (cf.~Figures~\ref{fig_C2C3C4} and \ref{fig_C5}), the reader might think that Condition~C4 implies Condition~C5.  However this is not true because it could be that the element $p$ in Condition~C4 equals $l$ in Condition~C5. 

\begin{figure}[h]
\begin{center}
\begin{tikzpicture}[yscale=1.5,xscale=1.5]
  \tikzstyle{vertex}=[circle,draw,fill=white, scale=0.3]
    
\begin{scope}[xshift=5.5cm]
  \node (i) at (0,0) [vertex,label=left:$i$] {};
  \node (j) at (0,1) [vertex,,label=left:$j$] {};
  \node (k) at (0,2) [vertex,label=left:$k$] {};
  \node (l) at (0,3) [vertex,label=left:$l$] {};
  \draw[->-=.7] (i) -- (j);
  \draw[->-=.7] (j) -- (k);
  \draw[->-=.7] (k) -- (l);
  \draw[->-=.7, dashed] (i) to [bend right] (k);
  \draw[->-=.7, dashed] (j) to [bend right] (l);
  
  \node (r) at (1,1) [vertex,label=above:$r$] {};
  \draw[->-=.7, dashed] (i) -- (r);
  \draw[->-=.7, dashed] (r) -- (k);
  
  \node (s) at (1,2) [vertex,label=below:$s$] {};
  \draw[->-=.7, dashed] (j) -- (s);
  \draw[->-=.7, dashed] (s) -- (l);
  
  \node (t) at (2,1.5) [vertex,label=right:$t$] {};
  \draw[->-=.7, dashed] (i) -- (t);
  \draw[->-=.7, dashed] (t) -- (l);
  
  \node (u) at (3,1) [vertex,label=right:$u$] {};
  \node (v) at (3,2) [vertex,label=right:$v$] {};
  \draw[->-=.7, dashed] (i) -- (u);
  \draw[->-=.7, dashed] (u) -- (v);
  \draw[->-=.7, dashed] (v) -- (l);
\end{scope}
  
\end{tikzpicture}
\end{center}
\caption{Condition~C5 on primary linear inequalities for $\pso n$. \label{fig_C5}}
\end{figure}
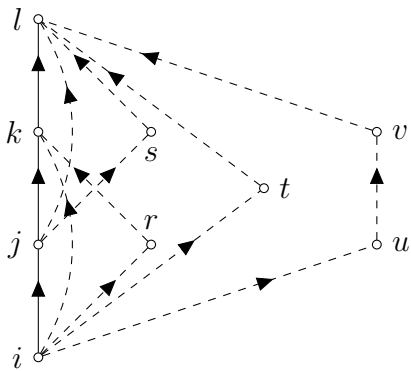

\begin{theo}\label{theo_valid_ineq}
Assume $A$ and $B$ are disjoint subsets of $\Arcs{n}$. 
Then the primary linear inequality
\begin{equation}
\label{eqn_general_valid}
\sum_{a \in A} x_{a} - \sum_{b \in B}x_b \leqslant 1
\end{equation}
is valid for $\pso n$ if and only if Conditions~C1, C2, C4 and C5 are satisfied. Moreover, the four conditions are logically independent.
\end{theo}

Before presenting the proof of Theorem~\ref{theo_valid_ineq} we establish a lemma which provides information on the semiorders $S$ whose characteristic vector $\chi_S$ invalidates  Equation~\eqref{eqn_general_valid}. 
\begin{lemm}\label{Lem_ex_11}
Assume $A$ and $B$ are disjoint subsets of $\Arcs{n}$ which satisfy Conditions~C1 and C2. Then for any semiorder $S$ on $\Vset{n}$ such that
\begin{equation}\label{eqn_lemme}
\sum_{a \in A} \chi^{\RR}_{a} - \sum_{b \in B} \chi^{\RR}_b \geqslant 1
\end{equation}
there holds $|S \cap A| \in \{1,2,3,4\}$. Moreover, if
\begin{equation}\label{eqn_lemme_2}
\sum_{a \in A} \chi^{\RR}_{a} - \sum_{b \in B} \chi^{\RR}_b > 1
\end{equation}
then
\begin{enumerate}
\item[\quad$(\alpha$)] either $S \cap A$ forms a path of length $2$  
and $S \cap B = \es$
\item[\quad$(\beta$)] or $S \cap A$ forms a path of length $3$, say $S \cap A =\{(i,j),(j,k),(k,l)\}$, 
and moreover $S \cap B = \{(i,l)\}$.
\end{enumerate}
\end{lemm} 

\begin{proof} 
If two pairs $(i,j)$ and $(k,l)$ in $S \cap A$ are disjoint, then by Condition~C2 both $(i,l)$ and $(k,j)$ belong to $B$ and because $S$ is a semiorder, at least one of them belongs to $S$.  By Condition~C1 such a pair $(i,l)$ or $(k,j)$ found to be in $S \cap B$ comes from a unique set of disjoint pairs $(i,j)$, $(k,l)$.  Hence
$$
\sum_{a \in A} \chi^{\RR}_{a} - \sum_{b \in B} \chi^{\RR}_b  
\quad\leqslant\quad 
|S \cap A| - |\{\text{sets of two disjoint pairs in } S \cap A\}|.
$$
Now let $|S \cap A| = m$.  By Condition~C1, $(\Vset{n},A)$ is a PC-graph, and so also $(\Vset{n},S \cap A)$ is a PC-graph (which moreover has no cycle because $S$ is a semiorder).  Hence a pair in $S \cap A$ must be disjoint from at least $m-3$ other pairs in $S \cap A$.  We then derive from previous equation
$$
\sum_{a \in A} \chi^{\RR}_{a} - \sum_{b \in B} \chi^{\RR}_b  
\quad\leqslant\quad 
m - \frac{m(m-3)}{2} \quad=\quad \frac{-m(m-5)}{2}.
$$
The latter expression is positive only for $m \in \{1,2,3,4\}$.  This establishes the first assertion.

Next assume Equation~\eqref{eqn_lemme_2} holds. Then $m \geqslant 2$. For any two disjoint pairs in $S \cap A$, we have shown in the first part of the proof the existence of some pair in $S \cap B$ (with no two of the latter pairs being equal). 
Hence the PC-graph $(\Vset{n},S \cap A)$ has to be a path of length at least $2$, maybe plus isolated vertices.  
If the path has length $4$, say $S \cap A =\{(i,j)$, $(j,k)$, $(k,l)$ $(l,p)\}$, then Condition~C2 implies that $(i,l)$, $(j,p)$ and $(i,p)$ are in $B$, and because they are also in $S$, Equation~\eqref{eqn_lemme_2} cannot hold.  There remain two cases: $(\alpha$) $S \cap A$ is a path of length 2, and then $S \cap B=\es$ in order to make Equation~\eqref{eqn_lemme_2} valid; $(\beta$) $S \cap A$ is a path of length $3$, say $S \cap A =\{(i,j),(j,k),(k,l)\}$; then by Condition~C2 $(i,l)\in B$.  No other pair of $S$ can be in $B$.
\end{proof}

We now prove Theorem~\ref{theo_valid_ineq}.

\begin{proof} 
Independence of Conditions~C1, C2, C4 and C5. Table~\ref{theo_valid_ineq} provides, for each of the four conditions, two subsets $A$, $B$ of $\Arcs{n}$ which satisfy all conditions but the one considered.  The  elements $i$, $j$, $k$, $l$ involved are all distinct.
\begin{table}[h]
$$
\begin{array}{c@{\quad}|@{\quad}c@{\quad}|@{\quad}c}
\text{Condition} & A & B \\
\hline
\text{C1}   & \{(i,k),\;(i,\;l)\}        & \es \\
\text{C2} & \{(i,\;j),\;(k,\;l)\}        & \es \\
\text{C4}  &  \{(i,\;j),\;(j,\;k)\}       & \es \\
\text{C5}   &  \{(i,\;j),\;(j,\;k),\;(k,\;l)\} & \{(i,l),\; (l,\;k),\;(k,\;j),(j,i)\}
\end{array}
$$
\caption{The pairs $A$, $B$ used to prove independence of C1, C2, C4 and C5. \label{tab_independence}}
\end{table}

Necessity of Conditions~C1, C2, C4 and C5. We have already proved in  Theorem~\ref{thm_valid_P} that C1, C2 and C4 are necessary, so we no turn to C5.  We proceed by contradiction: if C5 does not hold at distinct elements $i$, $j$, $k$, $l$, we produce a semiorder on $\Vset{n}$ whose characteristic vector does not satisfy our primary linear inequality \eqref{eqn_general_valid}.  Let  
\begin{align}
T &= \{t\in\Vset{n} \setminus \{i,j,k,l\} \st (i,t) \notin B \text{ and } (t,l) \notin B \};\\
U &= \{ u \in\Vset{n} \setminus \{i,j,k,l\} \st (i,u) \in B\};\\
V &= \{v \in\Vset{n} \setminus \{i,j,k,l\} \st  (v,l) \in B \}.
\end{align}
Notice that $\Vset{n} =  \{i,j,k,l\} \cup T \cup U \cup V$, a disjoint union which we display in Figure~\ref{fig_soTUV}.
Because we assume that C5 does not hold, we have $(i,k)$, $(j,l)\notin B$. For the same reason, $(u,k) \notin B$ and $(u,l) \notin B$ for any $u$ in $U$.  Similarly, for any $v$ in $V$, we have $(i,v) \notin B$ and $(j,v) \notin B$.  Finally, we have $(u,v) \notin B$ for each $u$ in $U$ and each $v$ in $V$.

Figure~\ref{fig_soTUV} displays the Hasse diagram of a semiorder $\RR$ which does not satisfy Equation~\eqref{eqn_general_valid}.  To show that $R$ is indeed a semiorder, we indicate an interval representation of constant length $3$ by providing the initial endpoints of the representing intervals:
$$
\begin{array}{c@{\qquad}c@{\qquad}c@{\qquad}c@{\quad}@{\quad}c@{\quad}cc}
i & j & k &  l & t\in T & u \in U & v \in V\\[2mm]
0 & 4 & 8 & 12 &    6   &    3    &     9
\end{array}
$$
The value at $\chi_{\RR}$ of the left-hand side of Equation~\eqref{eqn_general_valid} equals at least 2, because the pairs $(i,j)$, $(j,k)$ and $(k,l)$ of $\RR$ contribute a $+1$ while only $(i,l)$ can contribute a $-1$. 

\begin{figure}
\begin{center}
\begin{tikzpicture}[yscale=1.5,xscale=2]
\tikzstyle{vertex}=[circle,draw,fill=white, scale=0.3]

  \node (i) at (0,0) [vertex,label=right:$i$] {};
  \node (j) at (0,1) [vertex,,label=left:$j$] {};
  \node (k) at (0,2) [vertex,label=left:$k$] {};
  \node (l) at (0,3) [vertex,label=right:$l$] {};
  \draw (i) -- (j) -- (k) -- (l);

\node (t1) at (-2,1.5) [vertex] {};
\node at (-1.5,1.5){$\dots$};
\node (t2) at (-1,1.5) [vertex] {};
\draw (i) -- (t1) -- (l);
\draw (i) -- (t2) -- (l);

\node (u1) at (1,1) [vertex] {};
\node at (1.5,1){$\dots$};
\node (u2) at (2,1) [vertex] {};

\node (v1) at (1,2) [vertex] {};
\node at (1.5,2){$\dots$};
\node (v2) at (2,2) [vertex] {};

\draw (u1) -- (v1);
\draw (u1) -- (v2);
\draw (u1) -- (k);
\draw (u2) -- (v1);
\draw (u2) -- (v2);
\draw (u2) -- (k);
\draw (j) -- (v1);
\draw (j) -- (v2);

\draw[rectangle, rounded corners] (-2.25,1.25) rectangle (-0.75,1.75);
\node at (-2.5,1.6) {$T$};

\draw[rectangle, rounded corners] (0.75,0.75) rectangle (2.25,1.25);
\node at (2.5,0.8) {$U$};

\draw[rectangle, draw, rounded corners] (0.75,1.75) rectangle (2.25,2.25);
\node at (2.5,2.2) {$V$};
      
\end{tikzpicture}
\end{center}
\caption{The Hasse diagram of the semiorder used in the proof of Theorem~\ref{theo_valid_ineq}. \label{fig_soTUV}}
\end{figure}
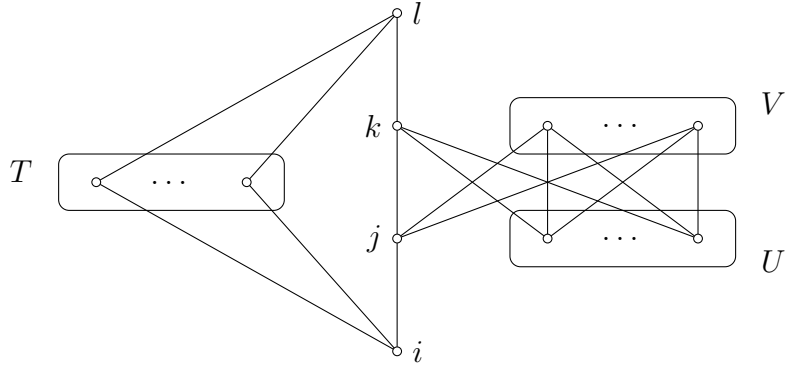

Sufficiency of Conditions~C1, C2, C4 and C5. Assume that the four conditions are satisfied.  We only need to show that any vertex of $\pso{n}$, in other words the characteristic vector $\chi_S$ of any semiorder $S$ on $\Vset{n}$, satisfies Equation~\eqref{eqn_general_valid}.  Suppose to the contrary that the equation is not satisfied  at $\chi_S$.  Then we have one of the two cases $(\alpha)$ and $(\beta)$ in Lemma~\ref{Lem_ex_11}.  

However, Case~$(\alpha)$ cannot occur: by C4, we have either $(i,k)$ in $S \cap B$ or, for some $p$ in $\Vset{n}\setminus\{i,j,k\}$, we have $(i,p)$ or $(p,k)$ in $S \cap B$.
Case~$(\beta)$ cannot neither occur because each requirement in C5 contradicts $S \cap B = \{(i,l)\}$.  Let us verify the assertion for Requirement~(v6), leaving (v1)--(v5) to the reader.  Because $S$ is a semiorder, $(i,u)$ or $(u,k)$ is in $S$.  In the first event, $(i,u)\in S \cap B$.  In the second event, $(u,v)\in S \cap B$ or $(v,l)\in S \cap B$.  We have reached a contradiction.
\end{proof}

For every PC-graph $(\Vset{n},A)$ with $A \neq \Vset{n}$, there is a subset $B$ of $\Arcs{n}$ such that Inequality~\eqref{theo_valid_ineq} is valid: it suffices to take  $B=\Arcs{n}\setminus A$.  Of course, the latter choice gives in general a very ``weak'' inequality: taking a smaller choice for $B$ gives a ``stronger'' inequality.  As a matter of fact, any ``least possible'' choice for $B$ produces a facet-defining inequality: this is the core of the next Theorem~\ref{theo_FDI_pso}.

\section{The Primary Facet-Defining Inequalities of the Semiorder Polytope}\label{sec_FDI_PSO} 

Theorem~\ref{theo_FDI_pso} below provides a characterization of those primary linear inequalities which define facets of the polytope $\pso{n}$.  Although unwieldy, the characterization directly leads to a polynomial-time algorithm for deciding whether a given primary linear inequality is facet-defining.   
It is moreover a practical tool for explicitly building primary facet-defining inequalities of $\pso{n}$ (see Theorem~\ref{theo_FDI_described}) and even, in principle, for listing all of them.

To state Theorem~\ref{theo_FDI_pso}, we first define in Example~\ref{ex_excep} the `exceptional' inequalities.  The example makes use of the following simple observation: any permutation of $\Vset n$ induces a permutation of coordinates in $\R^{\Arcs n}$, which leaves $\pso n$ invariant.

\begin{figure}[h]
\begin{center}
\begin{tikzpicture}[scale=5]
  \tikzstyle{vertex}=[circle,draw,fill=white, scale=0.3]
  
  \node (1) at (0,1) [vertex] {};
  \node (2) at (1,1) [vertex] {};
  \node (3) at (1,0) [vertex] {};
  \node (4) at (0,0) [vertex] {};
   \draw (1.west) node [above left] {$1$};
   \draw (2.east) node [above right] {$2$};
   \draw (3.east) node [below right] {$3$};
   \draw (4.west) node [below left] {$4$};
  
  \draw[->-=.7] (1) to [bend right=15] (2);
  \draw[->-=.7] (2) to [bend right=15] (3);
  \draw[->-=.7] (3) to [bend right=15] (4);
  \draw[->-=.7] (4) to [bend right=15] (1);
  
  \draw[->-=.7,dashed] (3) to [bend right=15] (2);
  \draw[->-=.7,dashed] (2) to [bend right=15] (1);
  \draw[->-=.7,dashed] (1) to [bend right=15] (4);
  \draw[->-=.7,dashed] (4) to [bend right=15] (3);
  
  \node (5) at (0.25,0.5) [vertex] {};
  \node (6) at (0.75,0.5) [vertex] {};
   \draw (5.west) node [left] {$5$};
   \draw (6.east) node [right] {$6$};

   \draw[->-=.7,dashed] (5) to (1);
   \draw[->-=.7,dashed] (5) to (2);
   \draw[->-=.7,dashed] (5) to (3);
   \draw[->-=.7,dashed] (5) to (4);
 
   \draw[->-=.5,dashed] (1) to (6);
   \draw[->-=.7,dashed] (2) to (6);
   \draw[->-=.7,dashed] (3) to (6);
   \draw[->-=.5,dashed] (4) to (6);
   \draw[->-=.7,dashed] (6) to (5);
   
         \node (7) at (1.5,0.5) [vertex] {};
          \node  at (1.75,0.5) {\textbf{...}};
           \node (n) at (2,0.5) [vertex] {};
            \draw (7.south) node [below] {$7$};
             \draw (n.south) node [below] {$n$};

\end{tikzpicture}
\end{center}
\caption{\label{fig_exceptional}
A graphical representation of the exceptional inequality from Example~\ref{ex_excep}.}
\end{figure}
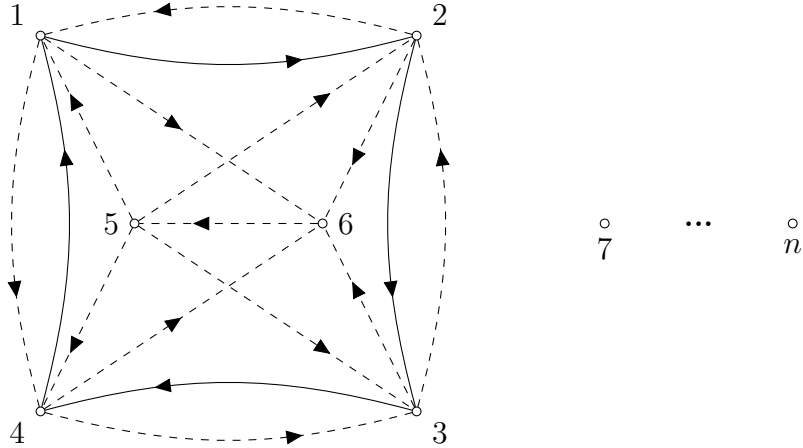

\begin{exam}[The exceptional inequality]\label{ex_excep}
Assume $n\geqslant6$ and let (see also Figure~\ref{fig_exceptional})
\begin{eqnarray}\label{AB_6}
A_6 &=& \{(1,2),\; (2,3),\; (3,4),\; (4,1)\},  \\
B_6 &=& \{(2,1),\; (3,2),\; (4,3),\; (1,4),  \\
  &&   \;\;(5,1),\; (5,2),\; (5,3),\; (5,4), \;(6,5),  \\
  &&   \;\;(1,6),\; (2,6),\; (3,6),\; (4,6)\}.
\end{eqnarray} 
A pair $A$, $B$ of subsets of $\Arcs n$ is \textsl{exceptional} when, possibly after a relabelling of the elements of $\Vset n$, 
there hold $A=A_6$ and $B=B_6$.  A primary linear inequality on $\R^{\Arcs n}$ is \textsl{exceptional} when, possibly after a relabelling of the elements of $\Vset n$, 
it takes the form
\begin{equation}\label{eqn_6}
\sum_{a \in A_6} x_{a} - \sum_{b \in B_6}x_b \leqslant 1.
\end{equation}  
All exceptional inequalities are valid for $\pso n$ (this results from Theorem~\ref{theo_valid_ineq} or can be checked directly).  None of them defines a facet, because the characteristic vector of any semiorder on $\Vset{n}$ containing the pair $(5,6)$ cannot give  equality in Equation~\eqref{eqn_6}, thus all vertices of $\pso n$ satisfying Equation~\eqref{eqn_6} with equality lie also in the face defined by $-x_{(5,6)} \leqslant 0$.  
\end{exam}

\begin{theo}\label{theo_FDI_pso} 
Let $A$ and $B$ be disjoint, nonempty subsets of $\Arcs{n}$.  The primary linear inequality 
\begin{equation}\label{eqn_general}
\sum_{a \in A} x_{a} - \sum_{b \in B}x_b \leqslant 1
\end{equation}  
defines a facet of $\pso n$ if and only if
\begin{enumerate}[\quad\rm (i)]
\item it is not exceptional;
\item $A$, $B$ satisfy Conditions~C1, C2, C4 and C5;
\item for any $c$ in $B$, replacing $B$ with $B\setminus\{c\}$ makes at least one of the Conditions~C1, C2, C4 and C5 becomes false.
\end{enumerate}
\end{theo}

When $A$ and $B$ are disjoint, nonempty subsets of $\Arcs{n}$, and moreover the condition after the ``if and only if'' in Theorem~\ref{theo_FDI_pso} holds, we say that the set $B$ is $A$-\textsl{minimal}.
In view of Theorem~\ref{theo_valid_ineq}, here is a rephrasing of the theorem (again assuming $A\neq \es \neq B$ and $A \cap B =\es$):  Inequality~\eqref{eqn_general} is facet-defining for $\pso n$ if and only it is valid and not exceptional but, for any $c $ in $B$, the following inequality is not valid:
\begin{equation}
\label{eqn_general_A_B_minus_c}
\sum_{a \in A} x_{a} - \sum_{b \in B\setminus\{c\}}x_b \leqslant 1.
\end{equation}
Remember from Theorem~\ref{thm_FDI_P}.2 that for $A=\es$ or $A$ a singleton, Equation~\eqref{eqn_general} never gives a facet-defining inequality.  For $B=\es$, it gives a facet-defining inequality exactly if $A=\{(i,j),(j,i)\}$ for some distinct elements $i$ and $j$.  According to Theorem~\ref{theo_FDI_pso}, the $n$-fence inequality from Example~\ref{exam_n_fence} defines a facet of $\pso n$.

Why do we need to mention the exceptional inequalities in Theorem~\ref{theo_FDI_pso}?  Because they satisfy all the other conditions after the ``if and only if'', but they do not define facets (as we saw in Example~\ref{ex_excep}).

\begin{proof}[Proof of Theorem \ref{theo_FDI_pso}]

Necessity.  When inequality \eqref{eqn_general} defines a facet, Conditions~C1, C2, C4 and C5 necessarily hold (Theorem~\ref{theo_valid_ineq}).  Working now by contradiction, assume that moreover \eqref{eqn_general_A_B_minus_c} is valid for some $c$ in $B$.  By adding to the last inequality the valid inequality $-x_c \leqslant 0$, we get Inequality~\eqref{eqn_general}.  Hence the latter cannot be facet-defining, a contradiction.

Sufficiency.  If Conditions~C1, C2, C4 and C5 hold, then by Theorem~\ref{theo_valid_ineq} inequality \eqref{eqn_general} is valid for $\pso n$.  To prove moreover that \eqref{eqn_general} defines a facet, we also assume that \eqref{eqn_general} is not exceptional and that $B$ is $A$-minimal.  For \eqref{eqn_general_A_B_minus_c} to be non-valid,  $|A|$ must be at least $2$.  If $|A| = 2$, our assumptions imply that \eqref{eqn_general} must be one of the four Axiomatic  Inequalities \eqref{eqn_axiomatic1}--\eqref{eqn_axiomatic4}, which we know to be facet-defining (Theorem~\ref{theo_basic_axiomatic}).

From now on we suppose $|A| \geqslant 3$.  By the Lifting Lemma (Theorem~\ref{theo_lifting}), we may moreover assume that any element from $\Vset{n}$ appears in at least one pair in $A \cup B$.   

Remember $\dim \pso{n}=n(n-1)$.  To show that the Inequality~\eqref{eqn_general} is facet-defining, we will produce $n(n-1)$ semiorders $S_{ij}$, one for each pair $(i,j)$ in $\Arcs{n}$, in such a way that their characteristic vectors are affinely independent and satisfy \eqref{eqn_general} with equality.   
Any pair $(i,j)$ in $\Arcs{n}$ is of exactly one of the following six types (the first type covers the situation where $(i,j)\in A$, the second, third and fifth $(i,j)\notin A \cup B$, the fourth and the sixth $(i,j)\in B$).  Figure~\ref{fig_four_types} illustrates the four first types.

\medskip

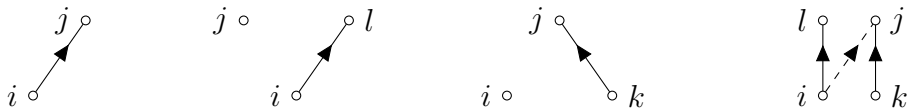
\begin{figure}[h]
\begin{center}
\begin{tikzpicture}[yscale=1,xscale=0.7]
  \tikzstyle{vertex}=[circle,draw,fill=white, scale=0.3]
  
\begin{scope}[xshift=-5cm]
  \node (i1) at (0,0) [vertex,label=left:$i$] {};
  \node (j1) at (1,1) [vertex,,label=left:$j$] {};
  \draw[->-=.7] (i1) -- (j1);
\end{scope}

\begin{scope}[xshift=0cm]
  \node (i2) at (0,0) [vertex,label=left:$i$] {};
  \node (j2) at (-1,1) [vertex,,label=left:$j$] {};
  \node (l2) at (1,1) [vertex,label=right:$l$] {};
    \draw[->-=.7] (i2) -- (l2);
\end{scope}

\begin{scope}[xshift=5cm]
  \node (i3) at (-1,0) [vertex,label=left:$i$] {};
  \node (j3) at (0,1) [vertex,,label=left:$j$] {};
  \node (k3) at (1,0) [vertex,label=right:$k$] {};
  \draw[->-=.7] (k3) -- (j3);

\end{scope}

\begin{scope}[xshift=10cm]
  \node (i4) at (0,0) [vertex,label=left:$i$] {};
  \node (j4) at (1,1) [vertex,label=right:$j$] {};
  \node (l4) at (0,1) [vertex,,label=left:$l$] {};
  \node (k4) at (1,0) [vertex,label=right:$k$] {};
  \draw[->-=.7,dashed] (i4) -- (j4);
  \draw[->-=.7] (k4) -- (j4);
  \draw[->-=.7] (i4) -- (l4);
\end{scope}

\end{tikzpicture}
\end{center}
\caption{\texttt{Types~1-4} of pairs $(i,j)$ in the proof of  Theorem~\ref{theo_FDI_pso}.\label{fig_four_types}}
\end{figure}

\noindent\texttt{Type~1\,:} $(i,j) \in A$. We then take the semiorder $S_{ij}=\{(i,j)\}$. 

\medskip

\noindent\texttt{Type~2\,:}
$(i,j) \notin A \cup B$ but $(i,l) \in A$ for some $l$ in $\Vset{n}$.  By Condition~C1, there can be only one such element $l$.  We let $S_{ij}=\{(i,j), (i,l)\}$.

\medskip

\noindent\texttt{Type~3\,:} $(i,j) \notin A \cup B$, $(i,l)\in A$ for no element $l$ and $(k,j) \in A$ for some element $k$ in $\Vset{n}$ (such a $k$ is unique by Condition~C1). We let $S_{ij}=\{(i,j), (k,j)\}$.

\medskip

\noindent\texttt{Type~4\,:} $(i,j) \in B$ and there exist some pairs $(i,l)$, $(k,j)$ in $A$ with $k\neq l$.  Notice that the latter pairs are unique.  We then let $S_{ij}=\{(i,l), (k,j), (i,j)\}$. 


\medskip

\noindent\texttt{Type~5\,:} $(i,j) \notin A \cup B$ and $(i,j)$ is not of \texttt{Types~2 or 3}.  This implies $(i,u), (v,j)\notin A$ for all $u$, $v$ in $\Vset{n}$.  Because of $|A| \geqslant 3$ and of Conditions~C1, there must exist $k$ and $l$ in $\Vset{n}\setminus\{i,j\}$ such that $(k,l) \in A$. 

Suppose first that it is possible to select such a pair $(k,l)$ with moreover $(i,l) \notin B$ or $(k,j) \notin B$.  We then take the semiorder $S_{ij}=\{(i,j), (k,l), z\}$, the pair $z$ being $(i,l)$ in the first case and $(k,j)$ in the second one. 
 
Suppose next no such choice of $(k,l)$ exists, that is $(k,j), (i,l) \in B$ for each pair $(k,l)$ in $A$ with $k,l \notin \{i,j\}$.
In this case, we claim that $A$ equals $\{(j,k),$ $(k,l),$ $(l,i)\}$ and that $B$ is either equal to $\{(j,i),$ $(i,l),$ $(l,k),$ $(k,j),$ $(j,t),$ $(t,i)\}$, for some $t \in \Vset{n} \setminus \{i,j,k,l\}$,
 or to $\{(j,i)$, $(i,l)$, $(l,k)$, $(k,j)$, $(j,u)$, $(u,v)$, $(v,i)\}$, for some $u$, $v \in \Vset{n} \setminus \{i,j,k,l\}$, with $u$ and $v$ distinct; moreover, in both cases, the resulting inequality is facet-defining.  The proof of the latter assertions being long, we  defer them to Lemmas~\ref{lem_technique} and \ref{lem_neuf}.

\medskip

\noindent\texttt{Type~6\,:} $(i,j)$ is of none of the previous \texttt{Types~1-5}.  Then necessarily $(i,j) \in B$.
By our basic assumption, $A$ and $B \setminus \{(i,j)\}$ do not satisfy Conditions~C1, C2, C4 or C5. Notice that Conditions~C1 remains true (because $A$ is not modified) and Condition~C2 also because $(i,j)$ is not of \texttt{Type~4}.  

Now the inequality
\begin{equation}
\label{eqn_general_A_B_minus_(i,j)}
\sum_{a \in A} x_{a} - \sum_{b \in B\setminus\{(i,j)\}}x_b \leqslant 1
\end{equation}
is not valid for $\pso n$, and so by Lemma~\ref{Lem_ex_11} there exists some semiorder $S$ on $\Vset{n}$ such that one of the two following holds:

\noindent\quad$\alpha$)~$S \cap A$ is a path of length $2$, and $S \cap (B \setminus\{(i,j)\}) = \es$.  Then we must have $S \cap B = \{(i,j)\}$, so we attach the semiorder $S$ to the pair $(i,j)$.

\noindent\quad$\beta$)~$S \cap A$ is a path of length 3, say $S \cap A =\{(u,v),(v,k),(k,l)\}$, and moreover $S \cap (B \setminus\{(i,j)\}) = \{(u,l)\}$.  Notice that the pair $(u,l)$ is of \texttt{Type~4}.  Also, $S \cap B = \{(u,l),(i,j)\}$. We attach the semiorder $S$ to the pair $(i,j)$.

At this point, we have attached to any pair $(i,j)$ in $\Arcs{n}$ some semiorder $S_{ij}$ on $\Vset{n}$---except in the singular situation as in the second part of \texttt{Type~5}, for which the conclusion results from the following two lemmas.  The characteristic vectors of all the $n(n-1)$ semiorders $S_{ij}$ satisfy Equation~\eqref{eqn_general} with equality.
They are moreover affinely independent. Indeed, the semiorder $S_{ij}$ contains the pair $(i,j)$ while all previously constructed semiorders do not contain that particular pair $(i,j)$; in other words, the characteristic vector of $S_{ij}$ has a $1$ in component $(i,j)$ while all the previous characteristic vectors have a $0$.
\end{proof}

The two lemmas below complete the handling of \texttt{Type~5} in the previous proof.

\begin{lemm}\label{lem_technique}
Assume that $A$ and $B$ are disjoint subsets of $\Arcs{n}$ which satisfy Conditions~C1 and C2. If $A$ contains pairs $(i,j)$, $(j,k)$ and $(u,v)$, $(v,w)$ such that
$$
\{(i,j), (j,k)\} \cap \{(u,v), (v,w)\} = \varnothing
$$
then $A$ and $B$ satisfy also Condition~C4 at the given elements $i$, $j$ and $k$.  Similarly, if $A$ contains pairs $(i,j)$, $(j,k)$, $(k,l)$ and $(u,v)$, $(v,w)$ such that
$$
\{(i,j), (j,k), (k,l)\} \cap \{(u,v), (v,w)\} = \varnothing
$$
then $A$ and $B$ satisfy also Condition~C5 at the given elements $i$, $j$, $k$ and $l$. 
\end{lemm}

\begin{proof}
By Condition~C1, $(\Vset{n},A)$ is a PC-graph. So $(i,j)$ and $(u,v)$ are disjoint, as well as $(j,k)$ and $(v,w)$.  Condition~C3 then implies that $(i,v)$ and $(v,k)$ are in $B$; this establishes Condition~C4 at $i$, $j$ and $k$.  The proof of the second assertion is similar.
\end{proof}

\begin{lemm}\label{lem_neuf} 
Consider nonempty, disjoint subsets $A$ and $B$ of $\Arcs{n}$ which satisfy Conditions~C1, C2, C4 and C5, with moreover $|A|\geqslant 3$,  $B$ being $A$-minimal and $A$, $B$ not exceptional (in the sense of Example~\ref{ex_excep}).  Suppose that there exists $(i,j)$ in $\Arcs{n}\setminus(A \cup B)$ such that 
\begin{enumerate}[\quad(I).~]
\item[\textnormal{(I)}.]
$(i,p)\notin A$ for all $p$ in $\Vset{n}$;
\item[\textnormal{(II)}.]
$(q,j)\notin A$ for all $q$ in $\Vset{n}$;
\item[\textnormal{(III)}.]
$(k,j), (i,l) \in B$ for each pair $(k,l)$ in $A$ disjoint from $(i,j)$.
\end{enumerate}
Then there exist $k$, $l \in \Vset{n}$ with $i$, $j$, $k$, $l$ pairwise distinct such that
$$
A = \{(j,k), (k,l), (l,i)\}
$$ 
and either of two cases: 
\begin{enumerate}[\quad\rm 1.~]
\item $B = \{(j,i),$ $(i,l),$ $(l,k),$ $(k,j),$ $(j,t),$ $(t,i)\}$ for some $t \in \Vset{n} \setminus \{i,j,k,l\}$, 
\item $B = \{(j,i)$, $(i,l)$, $(l,k)$, $(k,j)$, $(j,u),$ $(u,v),$ $(v,i)\}$ for some distinct $u$, $v$ in $\Vset{n} \setminus \{i,j,k,l\}$. 
\end{enumerate}
In both cases, Inequality~\eqref{eqn_general} defines a facet of $\pso n$. 
\end{lemm}

\begin{proof}
The specific inequalities in the two cases are valid in view of Theorem~\ref{theo_valid_ineq}.  To prove that they define facets, it suffices to exhibit respectively $5\cdot4 = 20$ and $6\cdot5 = 30$ semiorders with affinely independent characteristic vectors satisfying \eqref{eqn_general} with equality: in view of Theorem~\ref{theo_lifting}, it suffices to work with $\Vset{n}$ equal to $\{i$, $j$, $k$, $l$, $t\}$ or $\{i$, $j$, $k$, $l$, $u$, $v\}$, respectively.  We leave this to the reader.

We now show that if $A$ and $B$ satisfy the assumptions, then they are of one of the two latter forms.  Because of $|A| \geqslant 3$ and Conditions~C1 in Theorem~\ref{theo_valid_ineq}, there must exist $k$ and $l$ in $\Vset{n}\setminus\{i,j\}$ such that $(k,l) \in A$. 

First, let us infer the existence of some element $s$ in $\Vset{n}\setminus\{j,k,l\}$ such that $(l,s)\in A$. 
If no such element $s$ exists in $A$ we derive, from the present Assumption~(II) together with our $A$-minimality assumption (applied to the pair $(k,j)$ in $B$), that $(k,j)$ cannot be anything else than a pair as $(u,v)$ in Condition~C5: as shown in Figure~\ref{fig_ill_type_5}, there exist distinct elements $w$, $x$, $y$, $z$ in $\Vset{n}\setminus\{k,j\}$ such that the pairs $(w,x)$, $(x,y)$, $(y,z)$ are in $A$ and $(w,k)$, $(j,z)$ in $B$.  
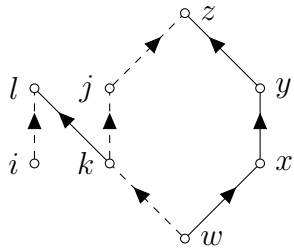
\begin{figure}[h]
\begin{center}
\begin{tikzpicture}[xscale=1,yscale=1]
  \tikzstyle{vertex}=[circle,draw,fill=white, scale=0.3]
  
   \node (i) at (-1,0) [vertex,label=left:{$i$}] {};
   \node (j) at (0,1) [vertex,label=left:{$j$}] {};
   \node (k) at (0,0) [vertex,label=left:$k$] {};
   \node (l) at (-1,1) [vertex,label=left:$l$] {};
  
   \node (w) at (1,-1) [vertex,label=right:{$w$}] {};
   \node (x) at (2,0) [vertex,label=right:{$x$}] {};
   \node (y) at (2,1) [vertex,label=right:{$y$}] {};
   \node (z) at (1,2) [vertex,label=right:{$z$}] {};
   
  \draw[->-=.7] (k) -- (l);
  \draw[->-=.7,dashed] (k) -- (j);
  \draw[->-=.7,dashed] (i) -- (l);

  \draw[->-=.7,dashed] (w) -- (k);
  \draw[->-=.7,dashed] (j) -- (z);

  \draw[->-=.7] (w) -- (x);
  \draw[->-=.7] (x) -- (y);
  \draw[->-=.7] (y) -- (z);

\end{tikzpicture}
\end{center}
\caption{First illustration for the proof of Lemma~\ref{lem_neuf}.\label{fig_ill_type_5}}
\end{figure}
Moreover, we must have $(k,l)$ and $(x,y)$ disjoint.  Then by Condition~C1, $(k,y) \in B$.  This shows that the deletion of $(k,j)$ never invalidates Condition~C5 (whatever the choices of $w$, $x$, $y$ and $z$), a contradiction with our present assumption in the statement.  We conclude that for some element $s$ in $\Vset{n}\setminus \{j,k,l\}$ the pair $(l,s)$ is in $A$.
Similarly (this time because $B\setminus\{(i,l)\}$ is $A$-minimal), there is some $r$ in $\Vset{n}\setminus \{i,k,l\}$ such that $(r,k) \in A$ (see Figure~\ref{fig_ill_type_5_bis}).
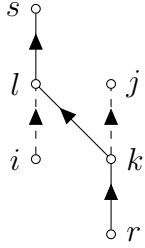
\begin{figure}[h]
\begin{center}
\begin{tikzpicture}[xscale=1,yscale=1]
  \tikzstyle{vertex}=[circle,draw,fill=white, scale=0.3]
  
   \node (i) at (-1,0) [vertex,label=left:{$i$}] {};
   \node (j) at (0,1) [vertex,label=right:{$j$}] {};
   \node (k) at (0,0) [vertex,label=right:$k$] {};
   \node (l) at (-1,1) [vertex,label=left:$l$] {};
  
   \node (r) at (0,-1) [vertex,label=right:{$r$}] {};
   \node (s) at (-1,2) [vertex,label=left:{$s$}] {};
   
  \draw[->-=.7] (k) -- (l);
  \draw[->-=.7,dashed] (k) -- (j);
  \draw[->-=.7,dashed] (i) -- (l);

  \draw[->-=.7] (r) -- (k);
  \draw[->-=.7] (l) -- (s);
   
\end{tikzpicture}
\end{center}
\caption{Second illustration for the proof of Lemma~\ref{lem_neuf}. Here the following equalities may occur: $i=s$, $j=r$, $r=s$.\label{fig_ill_type_5_bis}}
\end{figure}
Again by our present assumptions on the pair $(i,j)$, we know $r \neq i$ and $s \neq j$.  However $r$ might be equal to $j$ and $s$ might be equal to $i$.  We could also have $r=s$ in which case the two previous equalities cannot hold together.

\smallskip  

Suppose first that both equalities $r=j$ and $s=i$ hold (as in Figure~\ref{fig_t_5}).  
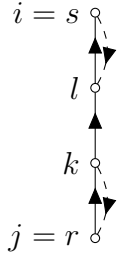
\begin{figure}[h]
\begin{center}
\begin{tikzpicture}[xscale=1,yscale=1]
  \tikzstyle{vertex}=[circle,draw,fill=white, scale=0.3]

   \node (j) at (0,0) [vertex,label=left:{$j=r$}] {};
   \node (k) at (0,1) [vertex,label=left:$k$] {};
   \node (l) at (0,2) [vertex,label=left:$l$] {};
   \node (i) at (0,3) [vertex,label=left:{$i=s$}] {};
  
  \draw[->-=.7] (j) -- (k);
  \draw[->-=.7] (k) -- (l);
  \draw[->-=.7] (l) -- (i);
  
  \draw[->-=.7, dashed] (i) to [bend left] (l);
  \draw[->-=.7, dashed] (k) to [bend left] (j);
\end{tikzpicture}
\end{center}
\caption{Third illustration for the proof of Lemma~\ref{lem_neuf}: when $i=s$ and $j=r$.\label{fig_t_5}}
\end{figure}
We prove that the Lemma holds by first etablishing $A = \{(j,k), (k,l), (l,i)\}$.  Suppose that $(x,y)$ is another pair of $A$.  Then $(i,y) \in B$ by Assumption~(III) and we will prove that replacing $B$ with $B \setminus \{(i,y)\}$ leaves a valid inequality, a contradiction with the $A$-minimality of $B$.  If replacing $B$ with $B \setminus \{(i,y)\}$ gives a nonvalid inequality, then either Condition~C4 is not satisfied or Condition~C5 is not satisfied by $A$ and $B \setminus \{(i,y)\}$. 
 
If Condition~C4 is not, then $(w,x) \in A$ for some $w$ distinct from $x,$ $y,$ $j,$ $k,$ $l,$ $i$. But this is impossible since then $(w,l)$ and $(l,y)$ are in $B$ by Condition~C2, and hence Condition~(C4) is satisfied by $A$ and $B \setminus \{(i,y)\}$..

So if $B \setminus \{(i,y)\}$ gives a nonvalid inequality, then Condition~C5 is not satisfied.  But then the three consecutive pairs of this condition must be all disjoint from $(j,k)$, $(k,l)$ and $(l,i)$, a contradiction by Lemma~\ref{lem_technique}.  This proves $A = \{(j,k)$, $(k,l)$, $(l,i)\}$.

Next, remember $(i,l)$, $(k,j)\in B$.  Moreover, by Condition~C2, $(j,i)$, $(l,k) \in B$.  We now apply Condition~C5 to the pairs $(j,k),$ $(k,l),$ $(l,i)$ to deduce that we are in one of the two cases of the statement of the lemma we are proving.  Actually, these two cases correspond to Requirements~(v5) and (v6) in Condition~C5.  Hence, we only have to show that if Requirements~(v1) to (v4) are satisfied we have a contradiction with the $A$-minimality of $B$.  Suppose Requirement~(v1) or Requirement~(v3) is satisfied.  Then $B$ is not $A$-minimal, because replacing $B$ with $B \setminus \{(i,l)\}$ still preserves Conditions~C1, C2, C4, C5.  The argument is similar for Requirements~(v2) and (v4), with the pair $(k,j)$ of $B$.  By the $A$-minimality of $B$, there is no further pair in $B$.  This concludes the proof when $r=j$ and $s=i$.

\smallskip

If $r \neq j$ or $s \neq i$, the situation must be as in one of the cases of Figure~\ref{ms_neq_ij}.  We deduce a contradiction between Assumption~(III) and  the $A$-minimality of $B$ in each case.
\begin{figure}[h]
\begin{center}
\begin{tikzpicture}[yscale=1]
  \tikzstyle{vertex}=[circle,draw,fill=white, scale=0.3]
  
  \node at (-0.2,3.7) {{Case 1}}; 
  \node (z_r) at (0,0) [vertex] {};
  \node (z_k) at (0,1) [vertex] {};
  \node (z_l) at (0,2) [vertex] {};
  \node (z_i) at (0,3) [vertex] {};
   \node (z_j) at (1,2) [vertex] {};
  
   \draw[->-=.7] (z_r) to (z_k);
   \draw[->-=.7] (z_k) to (z_l);
   \draw[->-=.7] (z_l) to (z_i);
   \draw[->-=.7, dashed] (z_k) to (z_j);
   \draw[->-=.7, dashed, bend left] (z_i) to (z_l);
     
   \draw (z_r.west) node [left] {$r$};
   \draw (z_k.west) node [left] {$k$};
   \draw (z_l.west) node [left] {$l$};
   \draw (z_i.west) node [left] {$s=i$};
   \draw (z_j.north) node [right] {$j$};
   
  
\begin{scope}[xshift=3.5cm]
  \node at (-0.4,3.7) {{Case 2}}; 
\begin{scope}[yshift=0.5cm]
  \node (z_r) at (-1,0) [vertex] {};
  \node (z_k) at (-1,2) [vertex] {};
  \node (z_l) at (0,2) [vertex] {};
  
  \node (z_i) at (1,0) [vertex] {};
  \node (z_j) at (0,0) [vertex] {};
  
   \draw[->-=.7] (z_r) to (z_k);
   \draw[->-=.7] (z_k) to (z_l);
   \draw[->-=.7] (z_l) to (z_r);
   \draw[->-=.7, dashed] (z_k) to (z_j);
   \draw[->-=.7, dashed] (z_i) to (z_l);
     
   \draw (z_r) node [below] {$r = s$};
   \draw (z_k) node [above] {$k$};
   \draw (z_l) node [above] {$l$};
   \draw (z_i) node [below] {$i$};
   \draw (z_j) node [below] {$j$};
\end{scope}
\end{scope}
 
 
\begin{scope}[xshift=7cm]

  \node at (0,3.7) {{Case 3}}; 
  \node (z_s) at (0,3) [vertex] {};
  \node (z_l) at (0,2) [vertex] {};
  \node (z_k) at (0,1) [vertex] {};
  \node (z_r) at (0,0) [vertex] {};
  
   \node (z_i) at (-1,1) [vertex] {};
   \node (z_j) at (1,2) [vertex] {};
  
   \draw[->-=.7] (z_r) to (z_k);
   \draw[->-=.7] (z_k) to (z_l);
   \draw[->-=.7] (z_l) to (z_s);
   
   \draw[->-=.7, dashed] (z_k) to (z_j);
   \draw[->-=.7, dashed] (z_i) to (z_l);
     
   \draw (z_r.west) node [left] {$r$};
   \draw (z_k.west) node [left] {$k$};
   \draw (z_l.west) node [left] {$l$};
   \draw (z_s.west) node [left] {$s$};
   \draw (z_j.east) node [right] {$j$};
   \draw (z_i.east) node [left] {$i$};
\end{scope}
   

\begin{scope}[xshift=10.5cm]
  \node at (-0.5,3.7) {{Case 4}}; 
  \node (z_s) at (0,3) [vertex] {};
  \node (z_l) at (0,2) [vertex] {};
  \node (z_k) at (0,1) [vertex] {};
  \node (z_r) at (0,0) [vertex] {};
  \node (z_i) at (-1,1) [vertex] {};
  
   \draw[->-=.7] (z_r) to (z_k);
   \draw[->-=.7] (z_k) to (z_l);
   \draw[->-=.7] (z_l) to (z_s);
   
   \draw[->-=.7, dashed, bend left] (z_k) to (z_r);
   \draw[->-=.7, dashed] (z_i) to (z_l);
     
   \draw (z_r.west) node [left] {$r=j$};
   \draw (z_k.west) node [left] {$k$};
   \draw (z_l.east) node [right] {$l$};
   \draw (z_s.east) node [right] {$s$};
   \draw (z_j.east) node [right] {$j$};
   \draw (z_i.west) node [left] {$i$};
\end{scope}
   
\end{tikzpicture}
\end{center}
\caption{Fourth illustration for the proof of Lemma~\ref{lem_neuf}: the four cases where $s \neq i$ or $r \neq j$.\label{fig_four_cases}}
\label{ms_neq_ij}
\end{figure}
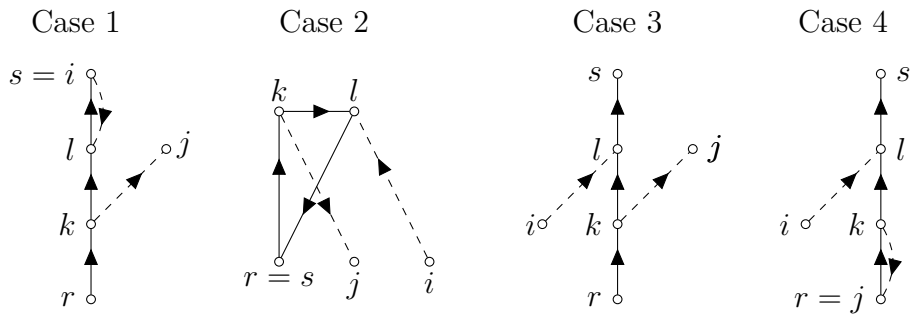
By Assumption~(III), the pair $(i,l)$ belongs to $B$; when $r \neq j$, the same assumption (after replacement of $(k,l)$ with $(r,k)$) gives $(r,j)\in B$.  

Consider first Case~2 of Figure~\ref{fig_four_cases}. We know that $B \setminus \{(r,j)\}$ is not valid. Either Condition~C4 or Condition~C5 is not satisfied by $B \setminus \{(r,j)\}$. By Lemma \ref{lem_technique}, $A$ does not contain two consecutive pairs $(u,v), (v,w)$ that are disjoint from $(r,k)$ and $(k,l)$
 (this would imply that $B \setminus \{(r,j)\}$ still satisfies Conditions~C1, C2, C4, C5, a contradiction). Hence $A$ does not contain three consecutive pairs $(u,v), (v,w), (w,z)$ with $u,v,w,z$ pairwise distinct, and we know that it is Condition~C4 that is not satisfied by $B \setminus \{(r,j)\}$. This implies then that $(j,l) \in B$ and then $B \setminus \{(i,l)\}$ is still valid. This gives the desired contradiction.

Cases~1 and 4 are similar, so we only treat Case~4. We claim that $B \setminus \{(l,j)\}$ is still valid (here, $(l,j)\in B$ because of Assumption~(III) and $(l,s)\in A$). If not, because $A$ and $B \setminus \{(l,j)\}$ satisfy Condition~C4, it must be Condition~C5 which is not satisfied and there exist pairs $(u,v)$, $(v,w)$, $(w,x)$ in $A$.  These pairs must be different from $(j,k)$ and $(k,l)$ and hence Condition~C5 must in fact be satisfied for $\{(u,v)$, $(v,w)$, $(w,x)\}$, a contradiction.

Finally, consider Case~3.  Exactly as we derived in the beginning of the present proof the existence of the element $s$ from the pair $(k,l)$ in $A$, we derive the existence of an element $t$ with $(s,t)\in A$ from the pair $(l,s)$ in $A$.  

Here is a general argument which we will be using several times.  Assume $A'$ and $B'$ satisfy Conditions~C1 and C2, where we write primes to make the distinction with our present notation.  Consider in Condition~C5 the given, distinct elements $i'$, $j'$, $k'$ and $l'$ and also a pair $(u',v')$ whose existence is asserted in Requirement~(v6).  If for some element $p'$ in $\Vset{n}\setminus\{j',k'\}$ we have $(u',p')$ in $A'$, then Requirement~(v3) holds (because $(u',p')$ and $(j',k')$ are disjoint pairs in $A'$).  Similarly, if for some element $q'$ in $\Vset{n}\setminus\{j',k'\}$ we have $(q',v')$ in $A'$, then Requirement~(v4) is met.

If $t\neq r$ we derive that $A$ and $B\setminus\{(i,l)\}$ satisfy Conditions~C1, C2, C4 and C5, in contradiction with our assumptions.  Indeed, $A$ and $B$ satisfy Conditions~C1, C2, C4 and C5; by Assumption~(I) and $(r,s)$, $(s,l)\in B$, we see that $A$ and $B\setminus\{(i,l)\}$ can only invalidate Condition~C5 with $(i,l)$ as $(u,v)$ in Requirement~(v6).  But the general argument from previous paragraph with $(k,l)\in A$ rules this out.  

If $t=r$, the pairs $(r,k)$, $(k,l)$, $(l,s)$, $(s,r)$ form a cycle of length $4$.  Assumptions~(I) and (II) imply $(i,p)$, $(p,j)\in B$ for any $p$ in $\{r$, $k$, $l$, $s\}$.  Thus we may take advantage of the cyclic symmetry w.r.t.\ $r$, $k$, $l$, $s$.
Notice also that $(l,k)$, $(k,r)$, $(r,s)$ and $(s,l)$ must be in $B$ because of Condition~C2.

We may not have $(p,i)$ in $B$ (for any $p$ in $\{r$, $k$, $l$, $s\}$), because otherwise $A$ and $B \setminus \{(p,j)\}$ still satisfy Conditions~C1, C2, C4 and C5 (use the general argument just above to check that $(p,j)$ is not as $(u,v)$ in Requirement~(v6); then, if for instance $p=k$, use Assumption~(II) and $(k,i)$, $(i,s)\in B$ to check that all Conditions~C1, C2, C4, C5 still hold).  Similarly, we may not have $(j,p)$ in $B$.
A similar argument shows $(s,k) \notin B$ (and by symmetry, $(r,l)$, $(k,s)$, $(l,r)\notin B$).

Next, consider the pair $(i,l)$.  By assumption $B$ is $A$-minimal, thus $A$ and $B\setminus\{(i,l)\}$ must invalidate Conditions~C1, C2, C4 or C5.  However, from Assumption~(I) and all the pairs we have obtained in $A$ and $B$ this is impossible except if $(i,l)$ is as $(v,l)$ in Requirement~(v6).  Thus there exists $u$ in $\Vset{n}\setminus\{s$, $r$, $k$, $l\}$ such that $(s,u)$, $(u,i) \in B$.  Similarly, $A$ and $B\setminus\{(s,j)\}$ must invalidate Conditions~C1, C2, C4 or C5, and this can only occur with the existence of $v$ in $\Vset{n}\setminus\{s$, $r$, $k$, $l\}$ such that $(j,v)$, $(v,l) \in B$.  Now if $v\neq i$ or $u \neq j$, we see that we cannot invalidate Requirement~(v6) both times,
a contradiction which completes Case~3. 
So we are left with $u=j$ and $v=i$.  Notice that because of Assumption~(III), there cannot be any further pair $(x,y)$ in $A$ (because $A$ and $B\setminus\{(i,y)\}$ would not invalidate Conditions~C1, C2, C4, C5). So we arrive at the exceptional example, (there may be isolated elements, not appearing in any pair of $A$ or of $B$).
\end{proof}

\vskip4mm

Condition~C1 implies that for a valid inequality 
\begin{equation}\label{eqn_general_valid_again}
\sum_{a \in A} x_{a} - \sum_{b \in B}x_b \leqslant 1
\end{equation} 
the graph $(\Vset{n}, A)$ is a PC-graph, that is, its components are isolated vertices, directed paths or directed cycles.  For many of the structural forms of the PC-graph $(\Vset{n}, A)$, there is only one set $B$ such that Equation~\eqref{eqn_general_valid_again} provides a facet-defining inequality.  We say that a pair $(i,l)$ in $\Vset{n}$ is $A$[C2]-\textsl{forced} if there exist elements $j$ and $k$ such that $i$, $j$, $k$ and $l$ are distinct and moreover $(i,j)$ and $(k,l)$ are in $A$.
Thus the $A$[C2]-forced pairs $(i,l)$ are exactly those that Condition~C2 forces to be in $B$.

\begin{theo}\label{theo_FDI_described}
Consider for disjoint, nonempty subsets $A$ and $B$ of $\Arcs{n}$, the inequality 
\begin{equation}\label{eqn_general_once_more}
\sum_{a \in A} x_{a} - \sum_{b \in B}x_b \leqslant 1,
\end{equation}  
and  suppose that the graph $(\Vset{n}, A)$ is a PC-graph whose  components meet at least one of the following six assumptions:
\begin{enumerate}[\quad\rm 1.]
\item at least one component contains two opposite pairs (of elements) in $A$ (and thus no other pair);
\item at least two components contain at least two pairs in $A$;
\item at least one component contains at least six pairs in $A$;
\item all nontrivial components contain exactly one pair of $A$.
\end{enumerate}
Then Equation~\eqref{eqn_general_once_more} describes a facet-defining inequality for $\pso n$ if and only if $B$ consists exactly of the $A$[C2]-forced pairs.
\end{theo}

\begin{proof}
First, remember that $(\Vset{n}, A)$ is a PC-graph exactly if Condition~C1.  Next, we note that if the PC-graph $(\Vset{n}, A)$ satisfies one of the assumptions 1--6, then the validity of Condition~C2 implies the validity of Conditions~C4 and C5 (this is easily established).  Then in the present situation Theorem~\ref{theo_FDI_pso} entails that Equation~\eqref{eqn_general_once_more} gives a facet-defining inequality if and only if $B$ contains the $A$[C2]-forced pairs but no other pair.
\end{proof}

To obtain a full list of all the primary facet-defining inequalities of $\pso{n}$, it remains to investigate the pairs $A$ and $B$ as in Theorem~\ref{theo_FDI_pso} for which the PC-graph $(\Vset{n}, A)$ does not satisfy any of the six assumptions of Theorem~\ref{theo_FDI_described}.  Such a  PC-graph has exactly one nontrivial component which contains more than one pair in $A$, and the number of its pairs equals $2$, $3$, $4$ or $5$ (moreover, in the first case the two pairs are not opposite).   Besides the $A$[C2]-forced pairs, $B$ must contain at least one other pair in order to make Conditions~C4 and C5 valid.  For any given value of $n$, it is in principle possible to list all such possible pairs $A$ and $B$ (say up to relabelling of elements).  However, even for small values of $n$, the task becomes quite tedious and from $n=6$ the number of examples is daunting. 

\section{Searching for the Primary Facet-Defining Inequalities of the Strict Weak Order Polytope}
\label{sect_Primary_FDI_PWO}
	
The strict weak order polytope $\pswo n$  \citep{Regenwetter_Davis_Stober_2012} is closely linked to still another polytope, the \textsl{weak order polytope} $\pwo n$ \citep[studied in][] {Doignon_Fiorini_2002, Fiorini_Fishburn_2004, Regenwetter_Davis_Stober_2008}.  A \textsl{weak order} (or \textsl{complete preorder}) on $\Vset{n}$ is a relation which is reflexive, transitive and total.  Thus the asymmetric part of a weak order is a strict weak order.  The \textsl{weak order polytope} is defined (again in $\R^{\Arcs{n}}$) by
$$
{\pwo n} =  \conv \left(\left\{ \chi^{W} \st W \text{ is a weak order on } \Vset n \right\}\right).
$$
In fact, the polytopes ${\pwo n}$ and ${\pswo n}$ are mutual images by the symmetry in the point $(\frac{1}{2}, \ldots, \frac{1}{2})$ (because the complement of a strict weak order is a weak order, and reciprocally).  Thus the inequality 
\begin{equation}\label{eqn_preorder_bis}
\sum_{(i,j) \in \Arcs{n}} \alpha_{ij}x_{ij} \leqslant \beta
\end{equation}
is facet-defining (resp.~valid) for ${\pswo n}$ if and only if 
\begin{equation}\label{eqn_preorder}
\sum_{(i,j) \in \Arcs{n}} -\alpha_{ij}x_{ij} \leqslant \beta -  \sum_{(i,j) \in \Arcs{n}} \alpha_{ij}
\end{equation}
is facet-defining (resp.~valid) for ${\pwo n}$.   
Our focus here is on primary FDIs for $\pswo n$.  Written as Equation~\eqref{eqn_preorder_bis}, such a primary FDI arises from an FDI of $\pwo n$ as in \eqref{eqn_preorder} with $\alpha_{ij} \in \{-1,0,1\}$ and $\beta - \sum_{(i,j) \in \Arcs{n}} \alpha_{ij}\in\{-1,0,1\}$.
A search in \cite{Doignon_Fiorini_2002} and \cite{Fiorini_Fishburn_2004} led to only nine primary FDIs of $\pswo n$, which we represent  in Figure~\ref{fig_seven_primary_pswo} (adhering to our usual conventions for such figures, also adding the value of the independent term). 

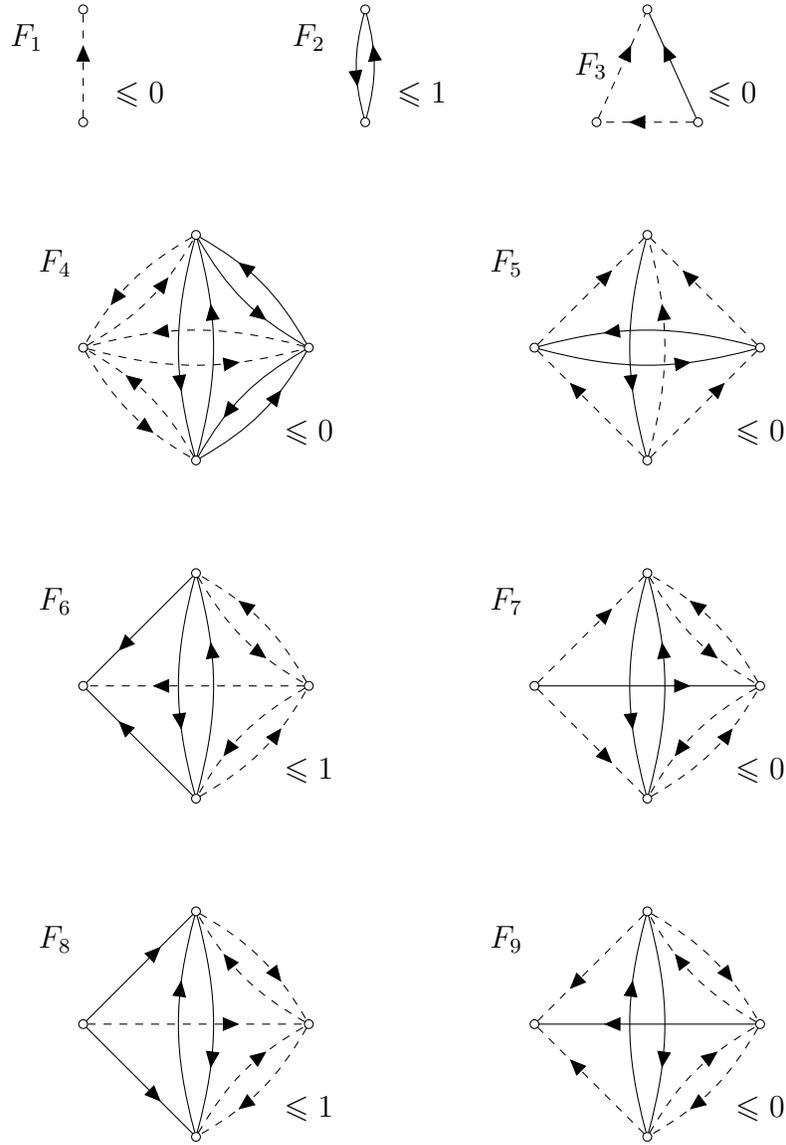
\begin{figure}[h!]
\begin{center}
\begin{tikzpicture}[scale=1.5]
  \tikzstyle{vertex}=[circle,draw,fill=white, scale=0.3]
  
  \node (1) at (0,0) [vertex] {};
  \node (2) at (0,1) [vertex] {};
  \draw[->-=.7,dashed] (1) to (2);
  \node  at (-0.5,0.75) {$F_1$};
  \node at (0.5,0.25) {$\leqslant 0$};

\begin{scope}[xshift=2.5cm]
    \node (2-1) at (0,0) [vertex] {};
    \node (2-2) at (0,1) [vertex] {};
 \draw[->-=.7,>= triangle 45] (2-1) to [bend right=15] (2-2);
 \draw[->-=.7,>= triangle 45] (2-2) to [bend right=15] (2-1);
    \node  at (-0.5,0.75) {$F_2$};
    \node at (0.5,0.25) {$\leqslant 1$};
\end{scope}

\begin{scope}[xshift=5cm]
   \node (3-1) at (0.45,0) [vertex] {};
   \node (3-2) at (-0.45,0) [vertex] {};
   \node (3-3) at (0,1) [vertex] {};
 \draw[->-=.7,>= open triangle 45,dashed] (3-1) to (3-2);
 \draw[->-=.7,>= open triangle 45,dashed] (3-2) to (3-3);
 \draw[->-=.7,>= triangle 45] (3-1) to (3-3);
   \node  at (-0.5,0.5) {$F_3$};
   \node at (0.75,0.25) {$\leqslant 0$};

\end{scope}

\begin{scope}[xshift=1cm,yshift=-2cm,scale=1]
  \node (4-1) at (1,0) [vertex] {};
  \node (4-2) at (0,1) [vertex] {};
  \node (4-3) at (-1,0) [vertex] {};
  \node (4-4) at (0,-1) [vertex] {};
\draw[->-=.7,>= open triangle 45,dashed] (4-2) to [bend right=15]  (4-3);
\draw[->-=.7,>= open triangle 45,dashed] (4-3) to [bend right=15]  (4-2);
\draw[->-=.7,>= open triangle 45,dashed] (4-1) to [bend right=15]  (4-3);
\draw[->-=.7,>= open triangle 45,dashed] (4-3) to [bend right=15]  (4-1);
\draw[->-=.7,>= open triangle 45,dashed] (4-3) to [bend right=15]  (4-4);
\draw[->-=.7,>= open triangle 45,dashed] (4-4) to [bend right=15]  (4-3);
\draw[->-=.7] (4-1) to [bend right=15] (4-2);
\draw[->-=.7] (4-2) to [bend right=15] (4-1);
\draw[->-=.7] (4-2) to [bend right=15] (4-4);
\draw[->-=.7] (4-4) to [bend right=15] (4-2);
\draw[->-=.7] (4-1) to [bend right=15] (4-4);
\draw[->-=.7] (4-4) to [bend right=15] (4-1);
   \node  at (-1.25,0.75) {$F_4$};
   \node at (1,-0.75) {$\leqslant 0$};

\end{scope}
  
\begin{scope}[xshift=5cm,yshift=-2cm,scale=1]
  \node (5-1) at (1,0) [vertex] {};
  \node (5-2) at (0,1) [vertex] {};
  \node (5-3) at (-1,0) [vertex] {};
  \node (5-4) at (0,-1) [vertex] {};
  \draw[->-=.7,>= open triangle 45,dashed] (5-1) to  (5-2);
  \draw[->-=.7,>= open triangle 45,dashed] (5-3) to  (5-2);
  \draw[->-=.7,>= open triangle 45,dashed] (5-4) to  (5-1);
  \draw[->-=.7,>= open triangle 45,dashed] (5-4) to  (5-3);
  \draw[->-=.7,>= open triangle 45,dashed] (5-4) to [bend right=15]  (5-2);
  
  \draw[->-=.7] (5-1) to [bend right=15] (5-3);
  \draw[->-=.7] (5-3) to [bend right=15] (5-1);
    \draw[->-=.7] (5-2) to [bend right=15] (5-4);
  \node  at (-1.25,0.75) {$F_5$};
 \node at (1,-0.75) {$\leqslant 0$};

\end{scope}

\begin{scope}[xshift=1cm,yshift=-5cm,scale=1]
  \node (6-1) at (1,0) [vertex] {};
  \node (6-2) at (0,1) [vertex] {};
  \node (6-3) at (-1,0) [vertex] {};
  \node (6-4) at (0,-1) [vertex] {};
  \draw[->-=.7,>= open triangle 45,dashed] (6-1) to [bend right=15] (6-2);
  \draw[->-=.7,>= open triangle 45,dashed] (6-2) to [bend right=15] (6-1);
  \draw[->-=.7,>= open triangle 45,dashed] (6-1) to [bend right=15] (6-4);
  \draw[->-=.7,>= open triangle 45,dashed] (6-4) to [bend right=15] (6-1);
  \draw[->-=.7,>= open triangle 45,dashed] (6-1) to  (6-3);
  
  \draw[->-=.7] (6-2) to [bend right=15] (6-4);
  \draw[->-=.7] (6-4) to [bend right=15] (6-2);
    \draw[->-=.7] (6-2) to (6-3);
        \draw[->-=.7] (6-4) to (6-3);

  \node  at (-1.25,0.75) {$F_6$};
 \node at (1,-0.75) {$\leqslant 1$};

\end{scope}

\begin{scope}[xshift=5cm,yshift=-5cm,scale=1]
  \node (7-1) at (1,0) [vertex] {};
  \node (7-2) at (0,1) [vertex] {};
  \node (7-3) at (-1,0) [vertex] {};
  \node (7-4) at (0,-1) [vertex] {};
  \draw[->-=.7,>= open triangle 45,dashed] (7-1) to [bend right=15] (7-2);
  \draw[->-=.7,>= open triangle 45,dashed] (7-2) to [bend right=15] (7-1);
  \draw[->-=.7,>= open triangle 45,dashed] (7-1) to [bend right=15] (7-4);
  \draw[->-=.7,>= open triangle 45,dashed] (7-4) to [bend right=15] (7-1);
  \draw[->-=.7,>= open triangle 45,dashed] (7-3) to  (7-2);
  \draw[->-=.7,>= open triangle 45,dashed] (7-3) to  (7-4);
  
  \draw[->-=.7] (7-2) to [bend right=15] (7-4);
  \draw[->-=.7] (7-4) to [bend right=15] (7-2);
  \draw[->-=.7] (7-3) to (7-1);

  \node  at (-1.25,0.75) {$F_7$};
 \node at (1,-0.75) {$\leqslant 0$};
 \end{scope}
 
\begin{scope}[xshift=1cm,yshift=-8cm,scale=1]
  \node (8-1) at (1,0) [vertex] {};
  \node (8-2) at (0,1) [vertex] {};
  \node (8-3) at (-1,0) [vertex] {};
  \node (8-4) at (0,-1) [vertex] {};
  \draw[->-=.7,>= open triangle 45,dashed] (8-2) to [bend left=15] (8-1);
  \draw[->-=.7,>= open triangle 45,dashed] (8-1) to [bend left=15] (8-2);
  \draw[->-=.7,>= open triangle 45,dashed] (8-4) to [bend left=15] (8-1);
  \draw[->-=.7,>= open triangle 45,dashed] (8-1) to [bend left=15] (8-4);
  \draw[->-=.7,>= open triangle 45,dashed] (8-3) to  (8-1);
  
  \draw[->-=.7] (8-2) to [bend left=15] (8-4);
  \draw[->-=.7] (8-4) to [bend left=15] (8-2);
  \draw[->-=.7] (8-3) to (8-2);
  \draw[->-=.7] (8-3) to (8-4);

  \node  at (-1.25,0.75) {$F_8$};
 \node at (1,-0.75) {$\leqslant 1$};
\end{scope}

\begin{scope}[xshift=5cm,yshift=-8cm,scale=1]
  \node (9-1) at (1,0) [vertex] {};
  \node (9-2) at (0,1) [vertex] {};
  \node (9-3) at (-1,0) [vertex] {};
  \node (9-4) at (0,-1) [vertex] {};
  \draw[->-=.7,>= open triangle 45,dashed] (9-1) to [bend left=15] (9-2);
  \draw[->-=.7,>= open triangle 45,dashed] (9-2) to [bend left=15] (9-1);
  \draw[->-=.7,>= open triangle 45,dashed] (9-1) to [bend left=15] (9-4);
  \draw[->-=.7,>= open triangle 45,dashed] (9-4) to [bend left=15] (9-1);
  \draw[->-=.7,>= open triangle 45,dashed] (9-2) to  (9-3);
  \draw[->-=.7,>= open triangle 45,dashed] (9-4) to  (9-3);
  
  \draw[->-=.7] (9-4) to [bend left=15] (9-2);
  \draw[->-=.7] (9-2) to [bend left=15] (9-4);
  \draw[->-=.7] (9-1) to (9-3);

  \node  at (-1.25,0.75) {$F_9$};
 \node at (1,-0.75) {$\leqslant 0$};
\end{scope}

\end{tikzpicture}
\end{center}
\caption{\label{fig_seven_primary_pswo}
Graphical representations of nine primary FDIs of $\pswo n$ built from FDI's of $\pwo n$ from \cite{Doignon_Fiorini_2002} and \cite{Fiorini_Fishburn_2004} (the independent term appears after the symbol ``$\leqslant$'').}
\end{figure}  

In several primary FDI's the independent term vanishes (in contrast to Theorem~\ref{thm_FDI_P}.1).
Moreover, Example~$F_4$ in Figure~\ref{fig_seven_primary_pswo} invalidates Conditions~C1 and C3 from Section~\ref{gen_fact}, while Example~$F_5$ invalidates Condition~C2.  In particular, $A$ does not always form a PC-graph---which complicates a lot the search for a classification of the primary FDIs of $\pswo n$.  We leave unsettled the problem  of characterizing the primary FDIs of $\pswo n$.

\section{Searching for the Primary Facet-Defining Inequalities of the Linear Ordering Polytope}
\label{sect_Primary_FDI_PLO}

The linear ordering polytope $\plo n$ has a richer history than the other order polytopes, including in psychology where it is often called the \textsl{binary choice polytope} (see for instance \citealp{Regenwetter_Dana_Davis_Stober_2010, Regenwetter_Dana_Davis_Stober_2011}). 
According to Equation~\eqref{eqn_dim_plo}, it is a polytope of dimension only $n(n-1)/2$.
Its affine hull is minimally described by all linear equations, for  distinct elements $i$, $j$,
\begin{equation}\label{eq_equality_PLO}
x_{ij} + x_{ji} = 1.
\end{equation} 
Moreover, for any two distinct elements $i$, $j$ and $k$, the two following inequalities define facets of $\plo n$:

\begin{align}
\label{eq_trivial_PLO}
-x_{ij} &\leqslant 0,\\
\end{align}

\begin{align}
\label{eq_trans_PLO}
x_{ij} + x_{jk} - x_{ik} &\leqslant 0;
\end{align}
they are respectively the \textsl{trivial inequality} and the \textsl{transitive inequality}.

\begin{theo}[\citealp{Dridi_1980}] 
Equations~\eqref{eq_equality_PLO}, \eqref{eq_trivial_PLO} and \eqref{eq_trans_PLO} form a linear description of the linear ordering polytope $\plo n$ if and only if $n \leqslant 5$.
\end{theo}

Hence we already know all the FDI's of $\plo n$ for $2 \leqslant n \leqslant 5$; moreover, for such values of $n$, there is a linear description of $\plo n$ which consists only of primary linear  equations.  The same holds for $n=6$, by results of \cite{Marti_Reinelt_2011}. 
Notice that the $3$-fence inequality defines a facet of $\plo 6$, and thus of $\plo n$.  When $3 \leqslant m$ and $2m \leqslant n$, \cite{Grotschel_Junger_Reinelt_1985a} and \cite{Cohen_Falmagne_1978}  \cite[see][]{Cohen_Falmagne_1990} independently established the same assertion for the $m$-fence and $\plo n$.

When searching for primary FDI's of $\plo n$ we have to take into account that a facet admits several descriptions (for instance, Equations~\eqref{eqn_3} and \eqref{eqn_4} define the same facet). 
Among the many descriptions of a facet, some might be primary.  Let us check this on a particular example extracted from a family in \cite{Grotschel_Junger_Reinelt_1985a}.

\begin{figure}[h!]
\begin{center}
\begin{tikzpicture}[scale=1.6]
  \tikzstyle{vertex}=[circle,draw,fill=white, scale=0.3]

  \node (1) at (26:2) [vertex] {};
  \node (2) at (77:2) [vertex] {};
  \node (3) at (128:2) [vertex] {};
  \node (4) at (180:2) [vertex] {};
  \node (5) at (-128:2) [vertex] {};
  \node (6) at (-77:2) [vertex] {};
  \node (7) at (-26:2) [vertex] {};
  
        \node at (26:2.15) {$1$};
        \node at (77:2.15) {$2$};    
        \node at (128:2.15) {$3$};    
        \node at (180:2.15) {$4$};    
        \node at (-128:2.15) {$5$};    
        \node at (-77:2.15) {$6$};    
        \node at (-26:2.15) {$7$};

  \node (8) at (26:1) [vertex] {};
  \node (9) at (77:1) [vertex] {};
  \node (10) at (128:1) [vertex] {};
  \node (11) at (180:1) [vertex] {};
  \node (12) at (-128:1) [vertex] {};
  \node (13) at (-77:1) [vertex] {};
  \node (14) at (-26:1) [vertex] {};
  
         \node at (26:0.8) {$8$};
        \node at (77:0.8) {$9$};    
        \node at (128:0.8) {$10$};    
        \node at (180:0.8) {$11$};    
        \node at (-128:0.8) {$12$};    
        \node at (-77:0.8) {$13$};    
        \node at (-26:0.8) {$14$};
  
   \node at (2.5,-1.8) {$\leqslant 17$};
 
 \draw [->-=.6] (1) -- (2);
 \draw [->-=.6] (2) -- (9);
 \draw [->-=.6] (9) -- (8);
 \draw [->-=.6] (8) -- (1); 
 \draw [->-=.6] (3) -- (2);
 \draw [->-=.6] (10) -- (3);
 \draw [->-=.6] (9) -- (10);
 \draw [->-=.6] (3) -- (4);
 \draw [->-=.6] (4) -- (11);
 \draw [->-=.6] (11) -- (10);
 \draw [->-=.6] (5) -- (4);
 \draw [->-=.6] (12) -- (5);
 \draw [->-=.6] (11) -- (12);
 \draw [->-=.6] (5) -- (6);
 \draw [->-=.6] (6) -- (13);
 \draw [->-=.6] (13) -- (12);
 \draw [->-=.6] (7) -- (6);
 \draw [->-=.6] (14) -- (7);
 \draw [->-=.6] (13) -- (14);
 \draw [->-=.6] (7) -- (8);
 \draw [->-=.6] (1) -- (14);
\end{tikzpicture}
\end{center}
\caption{\label{fig_Mobius}
The graphical description of a M\"obius inequality as in Example~\ref{exa_Mobius}.}
\end{figure}
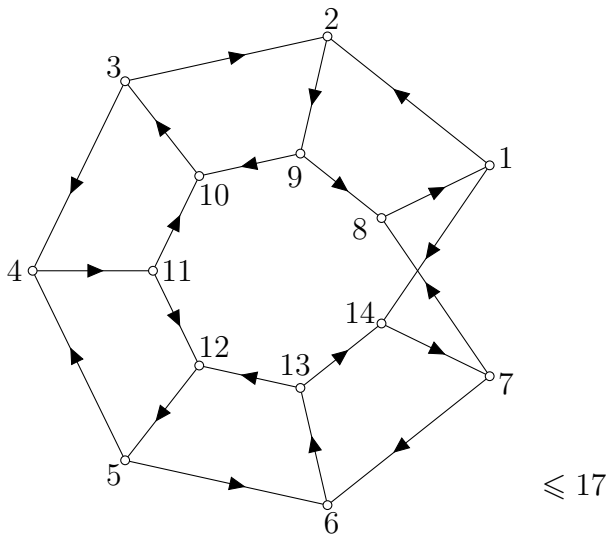

\begin{exam}[A M\"obius inequality] \label{exa_Mobius} 
Assume $n=14$ and let $A$ consist of the 21 pairs shown in Figure~\ref{fig_Mobius}.  Taking $B=\es$, we form the inequality
\begin{equation}\label{eqn_Mobius}
\sum_{a\in A} x_a \;\leqslant\; 17.
\end{equation}
The latter inequality defines a facet of $\plo n$ when $n \geqslant 14$ \citep{Grotschel_Junger_Reinelt_1985a}.  It fails to be primary only because of the independent term.  However we may remedy this by subtracting $16$, $17$ or $18$ equations~\eqref{eq_equality_PLO} from \eqref{eqn_Mobius} (selecting as many distinct, unordered pairs $\{i$, $j\}$).
\end{exam}

In the last example, a whole menagerie of primary FDI's of $\plo n$ results from the various possible choices of the pairs $\{i$, $j\}$.  Notice also that most of the resulting primary FDI's invalidate Conditions~C1 and C2;  moreover, some of them invalidate Conditions~C3, C4 and C5. 
Transforming a FDI of $\plo n$ into a primary FDI for the same facet can be done, of course, starting from many other equations than Equation~\eqref{eqn_Mobius}; for instance, all equations whose coefficients in front of variables take value $0$ or $+1$ apply (their independent term must be nonnegative and at most the number of $+1$, as seen by evaluating the left-hand side at a vertex in the facet), and even more equations do.  Consequently, the process produces a huge number of primary FDI's of the linear ordering polytope starting from known FDI's such as
\begin{enumerate}
\item M\"obius ladder inequalities, a family generalizing Equation~\eqref{eqn_Mobius} \citep{Grotschel_Junger_Reinelt_1985a},
\item $Z_k$-inequalities \citep{Reinelt_1985},
\item Paley inequalities \citep{Goemans_Hall_1996},
\item graphical inequalities \citep{Koppen_1995},
\item inequalities derived from nonorientable surfaces\citep{Fiorini_2006a},
\item a vast family containing further inequalities \citep{Fiorini_2006b}.
\end{enumerate}
In addition to these well-known classes, additional primary FDI's for $\plo n$ were obtained for $n=7$  \citep{Reinelt_1985}.
For a summary of FDI families known in year 2000, see \cite{Fiorini_thesis}.  The profusion and diversity of primary FDI's explain why we do not enter the classification enterprise for the linear ordering polytope.

\newpage 


\addcontentsline{toc}{section}{References}
\bibliographystyle{model5-names} 
\bibliography{bibli}

\end{document}